\documentclass[12pt,oneside]{article}
\newcommand{\Path}{}
\usepackage{ \Path Paper_style_2}
\usepackage{ \Path Keywords_WB}
\usepackage{ \Path Environments}
\usepackage{ \Path Category}

\usepackage{tikz-cd}
\usepackage{ \Path tikzcd2}
\usepackage{subcaption}

\DeclareMathOperator{\Null}{\mathcal{N}}
\DeclareMathOperator{\Nullity}{Nullity}

\newcommand{\main}{\mathcal{M}}
\newcommand{\base}{\mathcal{B}}
\newcommand{\inter}{\mathcal{I}}
\renewcommand{\REnv}[2]{ \mathcal{RE}_{#1} \left( #2 \right) }
\renewcommand{\LEnv}[2]{ \mathcal{LE}_{#1} \left( #2 \right) }
\newcommand{\JetEucZ}[1]{ C^{#1}-\textbf{Euc} }
\newcommand{\Poly}[1]{ \textbf{\text{Poly}}^{#1} }
\newcommand{\Jet}[1]{ \text{Jet}^{#1} }
\renewcommand{\Forget}{\bot}
\newcommand{\incl}{\iota}

\begin{document}
\title{The concept of null in general spaces and contexts}
\author{Suddhasattwa Das\footnotemark[1]}
\footnotetext[1]{Department of Mathematics and Statistics, Texas Tech University, Texas, USA}
\date{\today}
\maketitle

\begin{abstract} The notions of null-sets and nullity are present in all discourses of mathematics. They are based on the dual-pair of notions of "almost-every" and "almost none". A notion of nullity corresponds to a choice of subsets that one interprets as null or empty. The rationale behind this choice depends on the context, such as Topology or Measure theory. One also expects that the morphisms or transformations within the contexts preserve the nullity structures. To formalize this idea a generalized notion of nullity is presented as a functor between categories. A constructive procedure is presented by which an existing notion of nullity can be extended functorially to categories with richer structure. Nullity is thus presented as an arbitrary construct, which can be extended to broader contexts using well defined rules. These rules are succinctly expressed by right and left Kan extensions.
\end{abstract}

\section{Introduction} \label{sec:intro}

All discourses in Mathematics are rooted in some choice of a space. A space, loosely construed, is a set endowed with some structure. The mathematical properties to be examined are dictated by the choice of structure. For example the basic Euclidean spaces $\real^n$ could be studied as a vector space if the focus is on their linear structure, as a manifold if the focus is on their differential structure, or as metric spaces if the focus is on their in-built notion of distance. Any mathematical property is ultimately set theoretic and thus becomes synonymous with some subset of the space. The subset, called the \emph{characteristic set}, is simply the collection of points in the space which display the said property. One is often faced with the question of whether a property is typical or common, which is interpreted as a question of the typicality of its characteristic set. The equivalent version of this question is whether in a given space, a subset is empty, almost empty, almost full or full. This simple question has been of fundamental importance throughout the development of mathematics. The article presents the natural compositional structure in various notions of \emph{almost empty} and \emph{almost full}. 

Several landmark discoveries in mathematics are based on assumptions which do not hold uniformly over a space, but hold almost everywhere. In other words the sufficient conditions are violated on a null set. The two most familiar notions of null or \textit{nullity} are based on topology and measure respectively. Any topological space is naturally equipped with a notion of dense and nowhere-dense. Several applied fields such as \textit{Learning theory} rely on establishing that certain properties hold in a dense set. The most essential task of learning is choosing a \emph{hypothesis space}, and some of the key developments in Learning theory \cite[e.g.]{Hornik1990univ, HornikStinchcombeWhite1989multi, ChenChen1995universal, Das2023Lie, DasGiannakis2023harmonic} establish that certain hypothesis spaces are \emph{universal}, i.e., dense in a certain ambient functional space. Topological density also arises in spaces with more complicated nature, for example the space of dynamical systems. Dynamical systems represent systems evolving under a deterministic rule of transformation. In spite of the huge variety of phenomenon that can be seen, their study is made easier by the fact that there are some canonical dynamical systems \cite[e.g.]{DGJ_compactV_2018, PalisSmale1970structural, HirschEtAl1970ngbr, kieffer1980coding, AlpernPrasad2005towers, Das2023Koop_susp, Jewett1970prevalence} which are dense topologically. This is yet another intellectual merit of the notion of topological density. It justifies the study of some special cases as long as they are topologically dense.

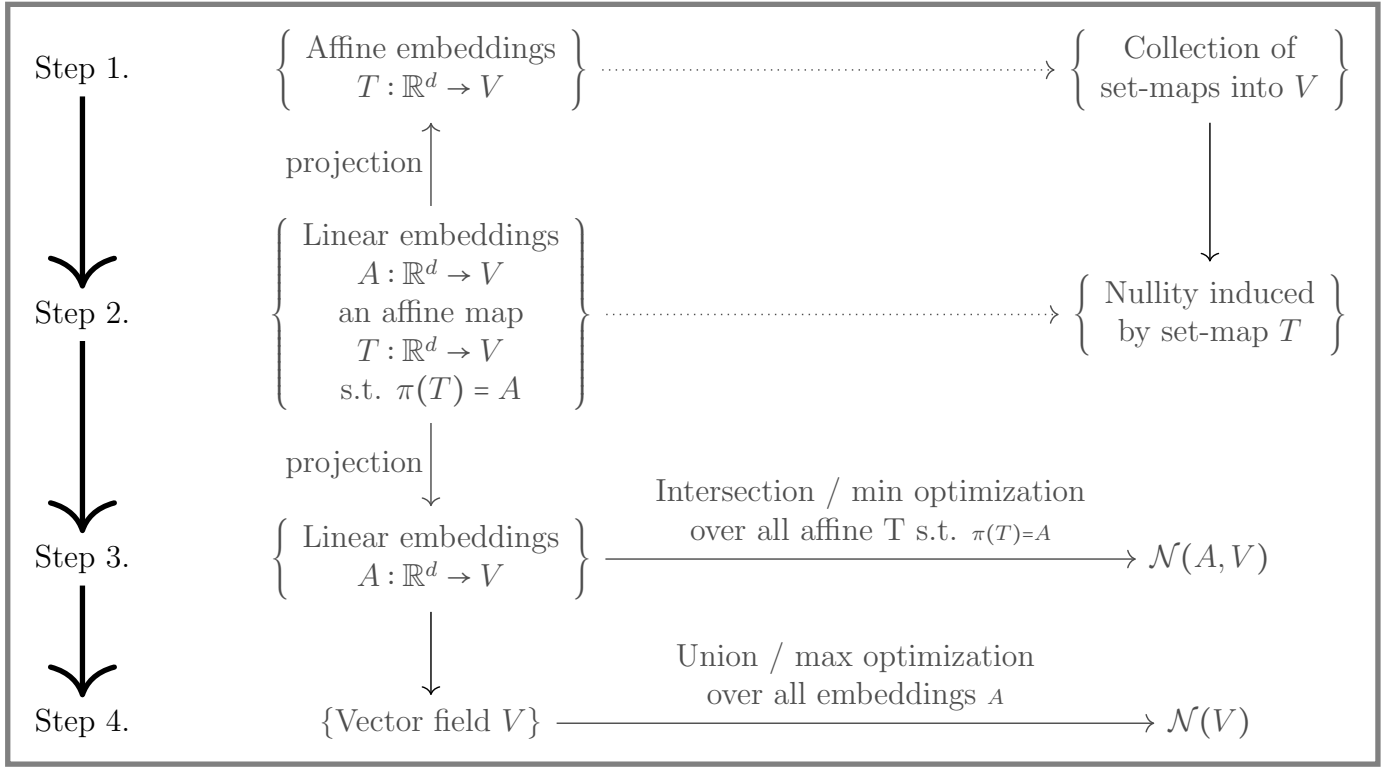
\begin{figure}[!t]
	\centering
	%
	\begin{tikzpicture} \node[draw, inner sep=5pt, draw=ChhaiD, line width=2pt] (box){
		\begin{tikzcd}[row sep = large, column sep = large]
			\mbox{Step 1.} \arrow[d, line width = 2] & \chhaiB{ \braces{ \begin{array}{c} \mbox{Affine embeddings} \\ T : \real^d \to V \end{array} } } \arrow[ChhaiB, dotted, rrrr] &&&& \chhaiB{ \braces{ \begin{array}{c} \mbox{Collection of} \\ \mbox{set-maps into } V \end{array} } } \arrow[d] \\
			\mbox{Step 2.} \arrow[d, line width = 2] & \chhaiB{ \braces{ \begin{array}{c} \mbox{Linear embeddings} \\ A : \real^d \to V \\ \mbox{an affine map} \\ T : \real^d\to V \\ \mbox{s.t. } \pi(T) = A  \end{array} } } \arrow[u, "\mbox{projection}", ChhaiB] \arrow[d, "\mbox{projection}"', ChhaiB] \arrow[ChhaiB, dotted, rrrr] &&&& \chhaiB{ \braces{ \begin{array}{c} \mbox{Nullity induced} \\ \mbox{by set-map } T \end{array} } } \\
			\mbox{Step 3.} \arrow[d, line width = 2] & \chhaiB{ \braces{ \begin{array}{c} \mbox{Linear embeddings} \\ A : \real^d \to V \end{array} } } \arrow[d] \arrow[ChhaiB]{rrrr}[name=n5]{\begin{array}{c} \mbox{Intersection / min optimization} \\ \mbox{over all affine T s.t. } \pi(T)=A \end{array}} &&&& \chhaiB{ \Null(A, V) } \\
			\mbox{Step 4.} & \chhaiB{ \braces{ \mbox{Vector field } V } } \arrow[ChhaiB]{rrrr}[name=n6]{\begin{array}{c} \mbox{Union / max optimization} \\ \mbox{over all embeddings } A \end{array}} &&&& \chhaiB{ \Null(V) }
		\end{tikzcd}
	}; \end{tikzpicture}
	%
	\caption{Construction of prevalence. The notion of prevalence, formulated in \cite{HSYprevalence1992} is one of the main inspirations for the analysis in this article. Prevalence is an extension of the notion of Lebesgue-null to infinite dimensional vector spaces. Thus prevalence assigns to any vector space $V$ a collection of sets labeled $\Null(V)$ which is closed under taking subsets, and which coincides with the usual Lebesgue zero measure sets if $V$ is finite dimensional. The construction of $\Null(V)$ goes through several steps, as presented in \eqref{eqn:prevalence:1} and \eqref{eqn:prevalence:6}. One first takes an arbitrary affine embedding $T$ of some $\real^d$ into $V$. such a $T$ is also a set-map and push-forwards the null sets of $\real^d$ into $V$. Next one takes the intersection of all these collections, as $T$ is varied with its linear part $A$ kept fixed. This leads to an $A$-dependent collection $\Null(A,V)$ of null-sets. Finally, one takes a union of all these $\Null(A,V)$ for various $A$s, to get the final nullity structure $\Null(V)$. This article shows how these steps are categorical in nature, and leads to a notion of nullity which is both invariant and an extension. The categorical approach allows this construction to be applied to several other contexts as well.}
	\label{fig:prevalence:1} 
\end{figure}

This useful notion of topological density is inadequate for many other situations. In the study of dynamical systems, one sees a peculiar feature in chaotic systems -- periodic points are topologically dense but statistically null \cite[e.g.]{DasYorke2020, DasSaddles2015, DasJim2017chaos}. An even simpler example is the density of rational numbers on the real line. A number picked at random is almost surely irrational. Thus topological full-ness might be irrelevant from a statistical or measure theoretic point of view. This prompts the formulation of the second notion of fullness or emptiness -- in terms of measure. One of the major breakthroughs in physics and mathematics was KAM theory \cite[e.g.]{Arnold1965, Arnold1963small, Herman1, Das2018Tongue}. It relies on the realization that the key property (quasiperiodicity) required to make conclusions about toral dynamics may not be universal but is measure theoretically full.

These two  notions -- \textit{topologically dense} and \textit{measure theoretically full} have been an essential component of several mathematical statements. They get interconnected whenever the reference measure is non-zero on the open sets of the topology. This is the case for the Haar measure on finite dimensional Lie groups. In such cases, being measure theoretically full implies being topologically dense. This is a simple and useful connection between two notions which are otherwise independent. The problem is that there is no natural Haar measure for infinite dimensional Lie groups, such as infinite dimensional vector spaces. 

The solution to this problem was the notion of \emph{prevalence} \cite{OttYorke_prevalence_2005, HSYprevalence1992}. Given any finite dimensional vector space $P$, we denote its Lebesgue volume measure by $\Leb_{P}$. Let $\calV$ be an infinite dimensional vector space, and $S$ a subset. Then $S$ is said to be \emph{shy} if there is a finite dimensional vector sub-space $P$ called a \emph{probe} such that 
\begin{equation} \label{eqn:prevalence:1}
	\Leb_{P} \paran{ P \cap \braces{ S+v } = 0 }, \quad \forall v\in \cal V .
\end{equation}
A subset is said to be \emph{prevalent} if its complement is shy. Thus shy and prevalent are analogs of almost empty and almost full, for infinite dimensional vector spaces. These notions have enabled results of fundamental importance to be stated and proved \cite[e.g.]{SauerEtAl1991, HuntKaloshin1997proj, sontag2003differential, fraysse2006smooth, HuntKaloshin1999reg, Kaloshin1997prevalence, Hunt1994dffrntl}. Note that according to \eqref{eqn:prevalence:1}, if a set $A$ is prevalent or shy, then so is any translate $A+v$ of $A$. This translation-invariance makes prevalence a natural definition for vector spaces. Prevalence offers an example of a new construction of nullity and fullness, which retains certain natural properties such as invariance under translation. This article re-examines this construction, and generalizes it to arbitrary settings. 

In spite of the naturality of prevalence, it is still inadequate to describe typical behavior in many situations. As pointed out in \cite{BerryDas2023learning, BerryDas2024review}, the existing notion of prevalence is not adequate for keeping track of null sets undergoing nonlinear transformation. One of the major discoveries in Dynamical systems theory was the phenomena of \emph{invariant graphs} \cite{Stark1999delay, Stark1999graphs} in skew product dynamical systems, which provide universal encodings of dynamical systems into Euclidean spaces. Invariant graphs arise from any pair comprising a dynamical system and a real-valued measurement on its phase-space. One of the most important conjectures is that almost every such pair can produce an encoding which is an embedding. There is no existing notion of nullity for such pairs, nor is there any tool to track the transformation of nullity under nonlinear transformations such as the operation of invariant-graphs.

Being typical or null are two sides of the same concept. Nullity is not a fixed and pre-conceived notion in spaces, but needs to be adapted and constructed. Ideally when one conceives of a notion of nullity, i.e., declares that a certain collection of sets in a space $V$ are the null sets, some logical rules are followed : 
\begin{enumerate} [(i)]
	\item subsets of a null set must be null too;
	\item if $V$ belongs to a class of similar objects, then the transformations or relations that bind these objects must preserve the respective nullity structures;
	\item if $V$ contains imprints or embeddings of structures such as Euclidean spaces, its new nullity structure should be an extension of those on the embedded images.
\end{enumerate}
The goal of this article is to present a very general construction of nullity, that preserves these properties. Note that nullity is a purely set-theoretic or lattice-theoretic notion and is independent of a choice of measure. We review again in Section \ref{sec:cat_cnstrct} how the collection of null sets according to \textit{prevalence} can be realized as the null-sets of various different choices of measures. The main discovery presented is the categorical nature of nullity, and how their extension to new contexts is essentially the categorical operation of \emph{Kan extensions}. The problem of invariant graphs described above pertains to Dynamical systems theory and it still remains open. The scope of this work which is primarily within Analysis, provides a clearer way forward to that problem.

The abstraction and generalization of nullity that we undertake is inspired by Sauer, Hunt and Yorke's construction of prevalence \cite{HSYprevalence1992}. The essential feature of prevalence are certain structural properties which are not limited to vector fields. These structural properties encode the embedding of the category of vector spaces within the category of affine spaces, and the projection from the latter to the former. We axiomatize such a relation in categorical language, in Assumptions \ref{A:1}, \ref{A:2} and \ref{A:3}. 

The categorical reconstruction of nullity also provides a separation of the ideas that go into nullity - there are some which are purely set-theoretic and some which are dependent on the context, which for example, could be vector spaces or manifolds. The former are universal and used in all constructs of nullity. Another contribution of this article is the category of $\Nullity$, which provides a concise mathematical definition of the universal set theoretic aspects of nullity.

\paragraph{Outline} We next take a closer look at the construction of prevalence in Section \ref{sec:prvlnc}. The definition will be restated in a manner that makes it suitable for immediate generalization. Next in Section \ref{sec:set} we present a general and broad definition of nullity, that uses the language of categories and functors. This allows a discussion of nullity which is very context independent. The categorical language developed is used to obtain a categorical redefinition of nullity and prevalence in Section \ref{sec:cat_cnstrct}. Next in Section \ref{sec:results} we analyze properties such as uniqueness, invariance and extendability of the nullity constructs. Finally in Section \ref{sec:examples} we present some examples of the applications of our theory. Section \ref{sec:proofs} contain some technical results and proofs.

\section{Deconstructing prevalence} \label{sec:prvlnc}

We now take a closer look at the concept of prevalence. Henceforth we shall use the term \emph{null} synonymously with \emph{shy}. Let us fix a vector space $V$ and a subset $S$ of $V$. We restate \eqref{eqn:prevalence:1} as
\[ S \mbox{ null} \quad\Leftrightarrow\quad \exists \begin{array}{c} \mbox{finite dim.} \\ \mbox{subspace} \end{array} P \,:\, \forall v\in \calV \,:\, \Leb_{P} \SetDef{ p\in P }{ p\in \braces{ S+v } } = 0 .\]
The translate $v$ is an invertible affine transform. Thus the above condition may be re-written as
\[ S \mbox{ null} \quad\Leftrightarrow\quad \exists \begin{array}{c} \mbox{finite dim.} \\ \mbox{subspace} \end{array} P \,:\, \forall v\in \calV \,:\, \Leb_{P} \SetDef{ p\in P }{ p + v \in S } = 0 .\]
Note that a probe is a finite dimensional subspace of $\calV$. The subspace $P$ may be interpreted as the image of a linear embedding $A:\real^d \to V$. For each $v\in V$, $x \mapsto Ax + v$ is an affine map from $\real^d \to V$. Given two vector spaces $X,Y$ let $\Affine(X;Y)$ denote the set of affine linear maps from $X$ to $Y$. Then the condition may be re-written as
\begin{equation} \label{eqn:prevalence:2} 
	S \mbox{ null} \quad\Leftrightarrow\quad \exists \begin{array}{c} \mbox{linear} \\ \mbox{embedding} \end{array} 
	\begin{tikzcd} \real^d \arrow[d, "A"'] \\ V \end{tikzcd} \,:\,
	\forall \begin{array}{c} \mbox{affine} \\ \mbox{map} \end{array}
	\begin{tikzcd} \real^d \arrow[d, "T"'] \\ V \end{tikzcd} 
	\mbox{ s.t. } \text{Lin}(T)=A \,:\,
	\Leb_{\real^d} T^{-1}(S) = 0 .
\end{equation}    
Equation \eqref{eqn:prevalence:2} restates \eqref{eqn:prevalence:1} in terms of inverse image under maps. A class of maps has been identified as $\Affine(X;Y)$ and nullity is in terms of inverse images under these maps. Let $\EucCat$ denote the collection of finite dimensional vector spaces. Given an $X\in \EucCat$ let $\Null(X)$ denote the subsets of $X$ which are null with respect to Lebesgue measure of $X$, which we have set as the natural measure for $X$. Let $\Affine \paran{ \real^d; V }$ denote the collection of affine maps from $\real^d$ to $V$.  Define 
\[\Null \paran{ V; T} := \SetDef{ \calS \in 2^{V} }{ T^{-1} (\calS) \in \Null \paran{\dom(T)} } , \quad T \in \Affine \paran{ \real^d; V } ,\; v\in V .\]
For every linear map $A$ and vector $v$ in its codomain, let $T_{A,v}$ denote the affine map $x \mapsto Ax + v$. We fix the notation $\Null \paran{ V; A, v} := \Null \paran{ V; T_{A,v}}$. Given any two vector space $X, X'$, let $\Hom_{\VectCat, \text{mono}}(X, X')$ denote the collection of all injective linear maps from $X$ to $X'$. The construction of prevalence follows the sequence :
\[ \Null \paran{ V; A} := \cap \SetDef{ \Null \paran{ V; T_{A, v}} }{ v\in V }, \quad \Null \paran{V} := \cup \SetDef{ \Null \paran{ V; A} }{ A \in \Hom_{\VectCat, \text{mono}} \paran{ \real^d; V } } .\]
In summary, the collection of null sets in $V$ is defined by the following max-min optimization :
\begin{equation} \label{eqn:prevalence:6}
	\boxed{ 
		\begin{array}{c} \Null(V) \\ {} \\ {} \\ {} \end{array} 
		\begin{array}{c} = \\ {} \\ {} \\ {} \end{array}
		\stackrel{\cup}{ \begin{array}{c} \text{Linear embedding} \\ A:\real^d \to V \\ {} \end{array} } 
		\stackrel{\cap}{ \begin{array}{c} \text{Affine embedding} \\ T:\real^d \to V \\ \proj(T) = A \end{array} } 
		\begin{array}{c} \Null \paran{ V; T} \\ {} \\ {} \\ {} \end{array} }
\end{equation}
The construction \eqref{eqn:prevalence:6} while being equivalent to \eqref{eqn:prevalence:1} presents our approach to the question of nullity and genericity. Now it is more apparent hoe the shy or null sets in an infinite dimensional space $\calV$ are built in terms of the null sets of finite dimensional spaces. The details of the construction are in the consecutive $\cup - \cap$ operation, akin to a $\max-\min$ operation This makes nullity a structural concept. Figure \ref{fig:prevalence:1} deconstructs this definition into logical steps. Nullity thus has the following elements to it : 
\begin{enumerate} [(i)]
	\item An infinite collection $\calC$ - which in this case is the collection of all vector spaces, finite and infinite.
	\item A concept of nullity for objects in a sub-collection - in this case the finite dimensional vector spaces.
	\item Nullity is a label attached to pairs $(S,X)$, with $X$ being an object of $\calC$ and $S$ a subset of $X$.
	\item The collection of objects in $\calC$ are bound to each other by some select relations - in this case affine maps.
	\item Nullity in a general object $X$ is in terms of pull-backs along these relations.
\end{enumerate}

All of this suggest categorical structures being present in all stages.

\paragraph{Category} The arrangement $\calC$ has objects and relations which can be composed. The composition of two affine maps is again affine. All this points to the mathematical property of a \emph{category} and nullity as a categorical or compositional property. A category is a bare structural definition that may be found in contexts of very different kinds. We present our categorical approach in the next section. The categorical approach is the logical choice for making a generalization of the notions of prevalent / typical, null / shy, in context different from linear spaces and affine maps, such as to differential maps and manifolds. 

A category $\calC$ is a collection of two kinds of entities : 
\begin{enumerate} [(i)]
	\item objects : usually representing different instances of the same mathematical construct;
	\item morphism : connecting arrows from one object to another which satisfy three properties --
	\item compositionality : given any three objects $a,b,c$ of $\calC$ and two morphisms $a \xrightarrow{f} b$ and $b \xrightarrow{g} c$, the morphisms can be joined end-to-end to create a composite morphism represented as $a \xrightarrow{g \circ f} b$;
	\item associativity : the composition of morphisms is associative;
	\item identity morphism : each object $a$ is endowed with a morphism $\Id_a$ called the \emph{identity} morphism, which play the role of unit element in composition.
\end{enumerate}

Given two points $x,y$ in the object set $ob(\calC)$, the collection of arrows from $x$ to $y$ is denoted as $\Hom(x;y)$. Note that this collection may be infinite, finite or even empty. The last criterion implies that for each $x$ $\Hom(x;x)$ has at least one member. Whenever there are multiple categories being discussed, one uses the notations $\Hom_{\calC}(x;y)$ or $\calC(x;y)$ to indicate that the morphisms are within the category $\calC$.

\paragraph{Examples of categories} One of the most fundamental categories is $\SetCat$, the category in which objects are sets up to a certain prefixed cardinality, and arrows are arbitrary maps. Similarly $\Topo$ denotes the category of topological spaces, with continuous maps as arrows. We denote by $\VectCat$ the category in which the objects are vector spaces and arrows are linear maps. The collection $\Affine$ that we have already defined has the same objects as $\VectCat$ but all affine maps as morphisms. Note that this includes the morphisms in $\VectCat$. This makes $\VectCat$ a \emph{subcategory} of $\Affine$. Suppose $\calU$ is any set. Then the power-set $2^{\calU}$ of subsets of $\calU$ is a category, in which the relations are the subset $\subseteq$ relations. Note that there can be only at most one arrow between any two objects $A,B$ of this category, which is to be interpreted as inclusion. Such categories are known as \emph{preorders}, and other examples are the category of ordered natural numbers, real numbers, open covers, and the concept of infinitesimal \cite[see]{Das2023CatEntropy}. Note that the usual notion of prevalence \eqref{eqn:prevalence:1}, \eqref{eqn:prevalence:6} is about objects and morphisms in the category $\Affine$. However, the label itself applies to arbitrary subsets of a vector space, a relation non-existent within $\Affine$. The subset relation is contained within some suitable chosen power set category. The notion of an inverse image also involves the category $\SetCat$, which is thus a third category that gets involved in the definition. To be able to work simultaneously with different categories and relate one with another, we need the notion of transformations that preserve categorical structure : \emph{functor}s.

\paragraph{Functor} Given two categories $\calC, \calD$, a functor $F:\calC\to \calD$ is a mapping between their objects along with the following properties : 
\begin{enumerate} [(i)]
	\item For each $x,y\in ob(\calC)$, there is an induced map $F : \Hom_{\calC}(x;y) \to \Hom_{\calD}( Fx; Fy)$. Thus arrows / morphisms between any pair of points get mapped into arrows between the corresponding pair of points in the image.
	\item $F$ preserves compositionality : given any three objects $a,b,c$ of $\calC$ and two morphisms $a \xrightarrow{f} b$ and $b \xrightarrow{g} c$, $F(g\circ f) = F(g) \circ F(f)$.
	\item $F$ preserves identity : $F(\Id_a) = \Id_{F(a)}$.
\end{enumerate}

In summary, a functor is a map between the object-sets that also preserves the underlying categorical structure. Categorical structure is essentially compositionality. The preservation of compositionality is expressed through the last two criterion. The language of categories and functors have helped create a \emph{synthetic} approach to a wide array of topics, such as homology \cite[e.g.]{Das2024hmlgy, BauerLesnick2020prsstnc}, Probability theory \cite[e.g.]{fritz2020Stoch, FritzEtAl2023repr}, learning theory \cite[e.g.]{Shiebler2020cluster, Shiebler2022kan}, and logic \cite[e.g.]{Lambek1989logic, Awodey1996}.  We are now ready to begin a categorical redefinition of nullity in multiple contexts.

\section{Set theoretic aspects of nullity} \label{sec:set}

Nullity is essentially a set-theoretic concept. Regardless of the context such as manifolds or vector spaces, the collection of null sets lies in the realm of sets. Recall that :

\begin{definition} [Down-set] \label{def:downset}
	Given a preorder $\calO$, a down-set is a complete sub-preorder, i.e., a collection $\tilde{\calO}$ of objects of $\calO$ such that if $b\in \tilde{\calO}$, $a$ is an object in $\calO$ and $a\leq b$, then $a$ belongs in $\tilde{\calO}$ too.
\end{definition}

The most common example of a preorder is the power set $2^V$ of the underlying space $V$. Definition \ref{def:downset} is essential since the concept of nullity is in fact a choice of a down-set within $2^V$. This collection is closed under unions, intersections and contains the empty set. 

\begin{definition} [Nullity for sets] \label{def:null:1}
	Given a set $A$, a nullity structure or concept of nullity for $A$ is a down-set of the power-set $2^{A}$ of $A$.
\end{definition}

Nullity is primarily a set theoretic aspect. We now extend it to objects in arbitrary categories. The following will be a standing assumption throughout the discussion :

\begin{Assumption} \label{A:1}
	There is a category $\main$,  to be interpreted as the main category, and a functor $\gamma : \main \to \SetCat$.
\end{Assumption}

The functor $\gamma$ acts as the bridge from $\main$ to $\SetCat$. This allows the definition :

\begin{definition} [Nullity for individual objects] \label{def:null:2}
	A nullity-structure for an object $V$ in the category $\main$ that satisfies Assumption \ref{A:1}, is a nullity structure for the set $\gamma(V)$.
\end{definition}

The next category enables a concise and categorical definition of the set-theoretic aspect of nullity.

\begin{definition} [Nullity category] \label{def:null:3}
	Let $\calS$ be a collection of sets. Then $\Nullity(\calS)$ denotes the category whose objects are 
	\[ \paran{ A, \calN_A } \;:\; A\in \calS, \, \calN_A \mbox{ is a nullity structure of } A . \]
	A morphism $\phi$ from an object $\paran{ A, \calN_A }$ into an object $\paran{ B, \calN_B }$  corresponds to a map $\phi: A\to B$ such that
	\begin{equation} \label{eqn:mrphsm:1}
		\phi(A) \in \calN_B, \quad \forall A \in \calN_A .
	\end{equation}
\end{definition}

It is routine to check compositionality and associativity in this category. The rule in \eqref{eqn:mrphsm:1} upholds the principle that a null set cannot be transformed into a non-null set, it must be transformed into another null set. The collection $\calS$ is often chosen to be the collection of all sets up to a certain cardinality. In that case $\calS$ is dropped from the notation and we just use $\Nullity$ for simplicity. 

\paragraph{Remark} An alternate way of defining morphisms is as set-theoretic maps $\phi: A\to B$ such that
\[ \phi^{-1}(B) \in \calN_A, \quad \forall B \in \calN_B . \]
This rule upholds the dual principle that a non-null set cannot be mapped into a null set. However the resulting categorical structure would not be conducive to our analysis. 

There is an obvious forgetful / projection functor $\proj : \Nullity(\calS) \to \left[ \text{Set} (\calS) \right] $ with $\left[ \text{Set} (\calS) \right]$ being the sub-category of $\SetCat$ spanned by the sets in the collection $\calS$. If $\calS$ is clear from the context, then we denote the functor as $\proj : \Nullity \to \text{Set} $. The true role of the functorial nature of $\gamma$ is brought to light from the next definition :

\begin{definition} [Nullity for categories] \label{def:null:4}
	Let $\base$ be a category satisfying Assumption \ref{A:1}. Then a nullity-structure for $\base$ is a functor $\Null : \base \to \Nullity$ such that the following commutation holds
	\begin{equation} \label{eqn:null:4}
		\begin{tikzcd}
			\base \arrow[drr, "\gamma"'] \arrow[rrrr, "\Null"] &&&& \Nullity \arrow[dll, "\proj"] \\
			&& \SetCat
		\end{tikzcd}
	\end{equation}
\end{definition}

Thus a nullity construct for the category $\base$ associates to each object $b\in \base$ the set $\gamma(b)$ along with a down-set $\calN_{b}$ of the power set of $\gamma(b)$. This assignment must be such that for every morphism $f : b\to b'$ in $\base$, the following rule is observed :
\begin{equation} \label{eqn:mrphsm:2}
	\gamma(f)( A ) \in \calN_{b'} , \quad \forall A \in \calN_{b}.
\end{equation}
We can find plenty of examples of Definition \ref{def:null:4}.

\begin{example} [Lebesgue nullity] \label{ex:1}
	Let $\EucCat$ be the category of finite vector spaces, and linear maps as morphisms. Let $\EucCat_{\text{mono}}$ denote the sub-category in which the morphisms are restricted to injective maps. Then the assignment of each finite dimensional space to its collection of Lebesgue zero-measure sets, is a nullity-construct in the sense of Definition \ref{def:null:4}.
\end{example}

In the next example and later we use $\ManCat{k}$ to denote the category of manifolds and $C^k$-differentiable maps. 

\begin{example} [Nowhere dense] \label{ex:2}
	Let $\ManCat{1}$ be  Then the assignment of each manifold to its collection of nowhere dense sets, is a nullity-construct in the sense of Definition \ref{def:null:4}.
\end{example}

\begin{example} [Measure spaces] \label{ex:3}
	Let $\MeasCat$ be the category of measurable spaces, and a morphism between two measure spaces $(\Omega, \Sigma, \mu)$ and $(\Omega', \Sigma', \mu')$ is a map $f:\Omega \to \Omega'$ which is measurable with respect to $\Sigma, \Sigma'$ and such that $f_* \mu$ is absolutely continuous with respect to $\mu'$. The assignment to each $(\Omega, \Sigma, \mu)$ the collection of sets in $\Sigma$ which have $\mu$-measure zero, is a nullity-construct in the sense of Definition \ref{def:null:4}.
\end{example}

For the next example, recall that a $G-\delta$ set is a countable intersection of open sets. A complement of a $G-\delta$ set is called an $F-\sigma$ set.

\begin{example} [F-sigma] \label{ex:4}
	The assignment to each topological space $\Omega$ its collection of $F-\sigma$ sets, leads to a nullity-construct on $\Topo$ in the sense of Definition \ref{def:null:4}.
\end{example}

\begin{example} [Wiener measure] \label{ex:Wiener}
	The Wiener measure was one of the first useful measures to be constructed on infinite dimensional spaces \cite{wiener1921average}. It is a measure on the space of continuous paths in any manifold. The collection of subsets of an Euclidean space with zero Wiener measure was shown to be invariant under translations \cite{CameronMartin1944Wiener}, and later more generally under various family of kernel integral operators \cite{CameronMartin1945Wiener}.
\end{example}

To complete our understanding of $\Nullity$ we recall one final general categorical concept.

\paragraph{Comma categories}  A general arrangement of categories and functors :
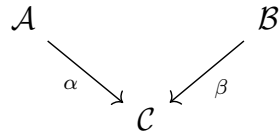
\[\begin{tikzcd}
	\mathcal{A} \arrow[rd, "\alpha"'] & & \mathcal{B} \arrow[ld, "\beta"] \\
	& \mathcal{C} & 
\end{tikzcd}\]
creates a special category called a \emph{comma category} $\Comma{\alpha}{\beta}$. Its objects are compound objects
\[ ob\paran{\Comma{\alpha}{\beta}} := \SetDef{ \paran{ a, b, \phi } }{ a\in ob(\calA), \, b\in ob(\calB), \, \phi \in \Hom_{\calC} \paran{ \alpha a ; \beta b } } , \]
and the morphisms comprise of pairs
$\SetDef{ (f,g) }{ f\in \Hom(D),\, g\in \Hom(E) }$ such that the following commutation holds :
\[\begin{tikzcd} \blue{ (a,\phi,b) } \arrow{r}{ (f,g) } & \akashi{ (a',\phi',b') } \end{tikzcd} \,\Leftrightarrow \,
\begin{tikzcd} a \arrow{d}{f} \\ a' \end{tikzcd}, \begin{tikzcd} b \arrow{d}{g} \\ b' \end{tikzcd}, \mbox{ s.t. }
\begin{tikzcd}
	\blue{ \alpha a } \arrow{r}{\alpha f} \arrow[blue]{d}{\phi} & \akashi{ \alpha a' } \arrow[Akashi]{d}{\phi'} \\
	\blue{ \beta b } \arrow{r}{\beta g} & \akashi{ \beta b' }
\end{tikzcd}\]
This category $\Comma{\alpha}{\beta}$ may be interpreted as connections between the functors $\alpha, \beta$, via their common codomain $\calC$. Comma categories contain as sub-structures, the original categories $\calA, \calB$, via the \emph{forgetful} functors
\[\begin{tikzcd} \calA & \Comma{\alpha}{\beta} \arrow{l}[swap]{\Forget^{\Comma{\alpha}{\beta}}_1} \arrow{r}{\Forget^{\Comma{\alpha}{\beta}}_2} & \calB \end{tikzcd}\]
If the functors $\alpha, \beta$ are clear from context, then the left and right forgetful functors will simply be denoted as $\Forget_1, \Forget_2$ respectively. Their action on morphisms in $\Comma{\alpha}{\beta}$ can be described as
\[ \begin{tikzcd} \blue{a} \arrow[blue]{d}{f} \\ \blue{a'} \end{tikzcd}
\begin{tikzcd} {} & & {} \arrow[ll, "\Forget_1"'] \end{tikzcd}
\begin{tikzcd}
	\blue{ \alpha a } \arrow{r}{\alpha f} \arrow[blue]{d}{\phi} & \akashi{ \alpha a' } \arrow[Akashi]{d}{\phi'} \\
	\blue{ \beta b } \arrow{r}{\beta g} & \akashi{ \beta b' }
\end{tikzcd}
\begin{tikzcd} {} \arrow[rr, "\Forget_2"] & & {} \end{tikzcd}
\begin{tikzcd} \akashi{b} \arrow[Akashi]{d}{g} \\ \akashi{b'} \end{tikzcd} \]
Comma categories prevail all over category theory and mathematics. If either $\calA$ or $\calB$ is taken to be $\star$ the trivial category with a single object , then the resulting comma categories are called left and right \emph{slice-categories} respectively. If $\calA = \calB = \calC$, then the comma category becomes the \emph{arrow-category}. The objects here are the arrows in $\calC$, and the morphisms are commutation squares. Comma, slice and arrow categories thus represent finer structures present within categories. Comma categories are used to represent various compound objects in mathematics \cite[e.g.]{Das2024slice, Das2023CatEntropy, DasSuda2024recon, DasSuda2025enrich}. The objects of a comma category are essentially morphisms, with their domain and codomain sourced from different categories. We next see how nullity from a component of a comma category leads to nullity for the entire comma category.

\paragraph{Nullity for commas} We now show that the concept of nullity can be extended if the set-theoretic interpretation allowed by Assumption \ref{A:2} is available. Suppose there is a nullity structure $\calN$ on some category $\calX$, and there is a second category $\calX'$ creating the arrangement shown below on the left :
\begin{equation} \label{eqn:ArrSet:1}
	\begin{tikzcd}
		\calX \arrow[dr, bend left=20, "\gamma"'] & & \calX' \arrow[dl, "\gamma'", bend right=20] \\
		& \SetCat
	\end{tikzcd} \imply 
	\begin{tikzcd} \Comma{\gamma}{\gamma'} \arrow{d}[swap, Shobuj]{\Null_{\gamma, \gamma'}} \\ \Nullity \end{tikzcd} , \quad 
	\begin{tikzcd} \gamma A \arrow[d, "f"'] \\ \gamma' A' \end{tikzcd} \;\mapsto\; 
	\paran{ \begin{array}{c} \gamma' A' \\ \SetDef{a'\subseteq \gamma' A'}{ (\gamma f)^{-1}(a') \in \Null \paran{ \gamma A } } \end{array} }
\end{equation}
Then one has a nullity construct on the comma category, as indicated on the right above. If two $\calX$ and $\calX'$ objects $A, A'$ respectively are joined set-theoretically via a map $f$, then the nullity functor $\Null_{\gamma, \gamma'}$ creates a nullity structure for $A'$ induced via $f$ from the pre-existing nullity structure on $A$. The action of this nullity functor on morphisms is shown below :
\[\begin{tikzcd}
	A \arrow[r, "f", Akashi] \arrow[d, "\alpha"'] & A' \arrow[d, "\beta"] \\
	B \arrow[r, "f'"', Holud] & B'
\end{tikzcd}
\begin{tikzcd} {} \arrow[rr, mapsto, "\Null_{\gamma, \gamma'}"] && {} \end{tikzcd}
\begin{tikzcd}
	\akashi{ \SetDef{b \subseteq \gamma(A) }{ f^{-1}(b) \in \Null(a) } } \arrow[d, "\beta"] \\ 
	\Holud{ \SetDef{b' \subseteq \gamma(A') }{ f'{-1}(b') \in \Null(a') } }
\end{tikzcd}\]
The nullity functor on the comma category remains bound to the nullity on $\calX$ and $\gamma'$ in the following manner :
\begin{equation} \label{eqn:ArrSet:2}
	\begin{tikzcd}
		\Comma{\gamma}{\gamma'} \arrow[rrr, "\Null_{\gamma, \gamma'}", Shobuj] \arrow[d, "\Forget_2"'] &&& \Nullity \arrow[d, "\Forget"] \\
		\calX' \arrow[rrr, "\gamma'"] &&& \SetCat
	\end{tikzcd}
\end{equation}
The functor $\Null_{\gamma, \gamma'}$ will serve as  the key tool for extending a pre-existing notion of nullity from a category $\calX$ to a different and possibly completely incompatible category $\calX'$. All that is required is a comma category to bind them, with $\SetCat$ serving as an intermediary.
We next begin the categorical axiomatization of nullity. The construction in \eqref{eqn:ArrSet:1} will be indispensable in this analysis.

\section{Categorical axiomatization} \label{sec:cat_cnstrct}

The construction starts with the assumption

\begin{Assumption} \label{A:2}
	There is a category $\base$ to be interpreted as the base-category, equipped with a notion of nullity, i.e.,  a functor $\Null : \base \to \Nullity$.
\end{Assumption}

The construction of nullity has a starting arrangement $\base$ (finite dimensional Euclidean spaces) with an existing notion of nullity (Lebesgue nullity). There is also a final arrangement $\main$ (infinite dimensional affine spaces) to which the notion of nullity has been extended to. There is also an intermediate arrangement $\inter$ (infinite dimensional vector spaces) is required, by which the $\base$-objects can be compared to the $\main$-objects. The arrangement $\inter$ has the same object complexity as $\main$, but less morphism complexity. They provide a common space in which $\main$ objects may be paired with $\base$ objects. These pairings are precisely the probes. The next two assumption formalizes this :

\begin{Assumption} \label{A:3}
	There are two categories $\inter$ and $\main$, to be interpreted as an intermediate category and the main category, along with functors creating the following arrangement :  
	\begin{equation} \label{eqn:A:3_4}
		\begin{tikzcd}
			\base \arrow[rr, "j_2"] && \inter \arrow[drr, bend right=5, "="] \arrow[rr, "j_1"] && \main \arrow[d, "\pi"] \\
			&& && \inter
		\end{tikzcd}
	\end{equation}
\end{Assumption}

\begin{figure}[!t]
	\centering
	\begin{subfigure}[t]{0.48\linewidth}
		\centering
		\begin{tikzpicture} \node[draw, inner sep=5pt, draw=ChhaiD, line width=2pt] (box){ 
				\begin{tikzcd}
					\base \arrow[rr, "j_2"] && \inter \arrow[drr, bend right=5, "="] \arrow[rr, "j_1"] && \main \arrow[d, "\pi"] \\
					&& && \inter
				\end{tikzcd}
			}; \end{tikzpicture}        
		\caption{ Assumption \ref{A:3} -- Diagram \eqref{eqn:A:3_4}.}
		\label{fig:A:3_4:1}
	\end{subfigure} 
	\begin{subfigure}[t]{0.48\linewidth}
		\centering
		\begin{tikzpicture} \node[draw, inner sep=5pt, draw=ChhaiD, line width=2pt] (box){ 
				\begin{tikzcd} [scale cd = 1]
					\EucCat_{\text{mono}} \arrow[r, "j_2", "\subset"'] & \VectCat_{\text{mono}} \arrow[dr, bend right=5, "="'] \arrow[r, "j_1", "\subset"'] & \AffineCat_{\text{mono}} \arrow[d, "\pi"] \\
					& & \VectCat_{\text{mono}}
				\end{tikzcd}
			}; \end{tikzpicture}
		\caption{ See Example \ref{ex:preval}.}
		\label{fig:A:3_4:2}
	\end{subfigure} 
	
	\begin{subfigure}[t]{0.48\linewidth}
		\centering
		\begin{tikzpicture} \node[draw, inner sep=5pt, draw=ChhaiD, line width=2pt] (box){ 
				\begin{tikzcd} [scale cd = 1]
					\EucCat_{\text{mono}} \arrow[r, "j_2", "\subset"'] & \Poly{1} \arrow[dr, bend right=5, "="'] \arrow[r, "j_1", "\subset"'] & \JetEucZ{1} \arrow[d, "\Jet{1}\rvert_{0}"] \\
					& & \Poly{1}
				\end{tikzcd}
			}; \end{tikzpicture}
		\caption{See Example \ref{ex:JetEuc:1}.}
		\label{fig:A:3_4:4}
	\end{subfigure} 
	\begin{subfigure}[t]{0.48\linewidth}
		\centering
		\begin{tikzpicture} \node[draw, inner sep=5pt, draw=ChhaiD, line width=2pt] (box){ 
				\begin{tikzcd} [scale cd = 0.9]
					\JetEucZ{k} \arrow[r, "j_2", "\subset"'] & \Poly{k} \arrow[dr, bend right=5, "="'] \arrow[r, "j_1", "\subset"'] & \JetEucZ{k+1} \arrow[d, "\Jet{k}\rvert_{0}"] \\
					& & \Poly{k}
				\end{tikzcd}
			}; \end{tikzpicture}
		\caption{See Example \ref{ex:JetEuc:k}.}
		\label{fig:A:3_4:5}
	\end{subfigure} 

	\begin{subfigure}[t]{0.48\linewidth}
		\centering
		\begin{tikzpicture} \node[draw, inner sep=5pt, draw=ChhaiD, line width=2pt] (box){ 
				\begin{tikzcd} [scale cd = 0.9]
					\JetEucZ{k} / _{\sim} \arrow[d, "j_2", "="'] \\ 
					\JetEucZ{k} / _{\sim} \arrow[dr, bend right=5, "="'] \arrow[r, "j_1", "\subset"'] & \ManCat{k}_{pt} \arrow[d, "\Jet{k}"] \\
					& \JetEucZ{k} / _{\sim}
				\end{tikzcd}
			}; \end{tikzpicture}
		\caption{See Example \ref{ex:man_pt:k}}
		\label{fig:A:3_4:6}
	\end{subfigure} 
	\begin{subfigure}[t]{0.48\linewidth}
		\centering
		\begin{tikzpicture} \node[draw, inner sep=5pt, draw=ChhaiD, line width=2pt] (box){ 
				\begin{tikzcd} [scale cd = 0.9]
					\ManCat{k}_{pt} \arrow[d, "\proj"'] \\ 
					\ManCat{k} \arrow[dr, bend right=5, "="'] \arrow[r, "="] & \ManCat{k} \arrow[d, "="] \\
					& \ManCat{k}
				\end{tikzcd}
			}; \end{tikzpicture}
		\caption{See Example \ref{ex:man_pt:k}}
	\end{subfigure} 
	
	\begin{subfigure}[t]{0.48\linewidth}
		\centering
		\begin{tikzpicture} \node[draw, inner sep=5pt, draw=ChhaiD, line width=2pt] (box){ 
				\begin{tikzcd}
					\base \arrow[rr, "="] && \base \arrow[drr, bend right=5, "="] \arrow[rr, "="] && \base \arrow[d, "="] \\
					&& && \base
				\end{tikzcd}
			}; \end{tikzpicture}
		\caption{Special case of Sub-figure (a) in which $\base = \inter = \main$. See Definition \ref{def:sat_null}.}
		\label{fig:A:3_4:3}
	\end{subfigure} 
	\begin{subfigure}[t]{0.48\linewidth}
		\centering
		\begin{tikzpicture} \node[draw, inner sep=5pt, draw=ChhaiD, line width=2pt] (box){ 
				\begin{tikzcd}
					\base \arrow[rr, "j_2"] && \main \arrow[drr, bend right=5, "="] \arrow[rr, "="] && \main \arrow[d, "="] \\
					&& && \main
				\end{tikzcd}
			}; \end{tikzpicture}
		\caption{Special case of Sub-figure (a) in which $\inter = \main$. See equations \eqref{eqn:def:veeNull:1}, \eqref{eqn:def:veeNull:2}.}
		\label{fig:A:3_4:7}
	\end{subfigure} 
	\caption{Instances of diagram \eqref{eqn:A:3_4}. One of our main assumptions is the arrangement of categories and functors presented in the pattern of \eqref{eqn:A:3_4}. The pattern provides a means of lifting a nullity structure from the base category $\base$ to the main category $\main$. The panels show several instances of this pattern.  Note that these diagrams are independent of the existing nullity structure on $\base$, presented in \eqref{eqn:null:4}. A construction of nullity thus involves two independent choices for realizing diagrams \eqref{eqn:null:4} and \eqref{eqn:A:3_4}. }
	\label{fig:A:3_4}
\end{figure}
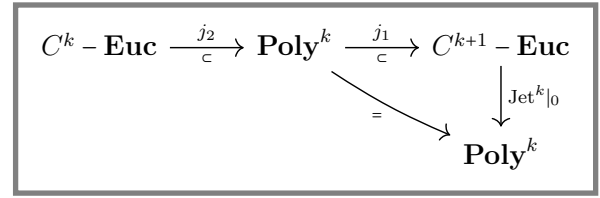
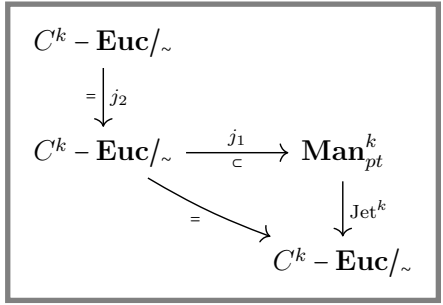
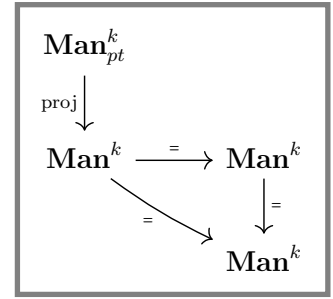
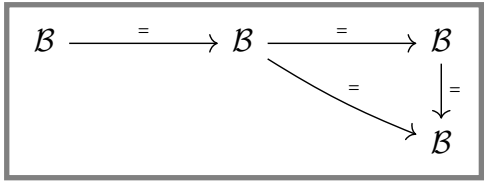

Figure \ref{fig:A:3_4} presents several instances of the abstract categorical diagram \eqref{eqn:A:3_4}. With this in mind we now re-examine the construction of prevalence, using the language of categories and functors. Now all the conclusions about the setup in Figure \ref{fig:A:3_4:2} will also hold for any setup satisfying Assumptions \ref{A:1}--\ref{A:3}.

\begin{enumerate}
	\item Consider any vector space $V$. It is an object of both $\Affine$ as well as $\VectCat$. We choose the former.
	\item A linear embedding $A : \real^d \to V$ corresponds to an object of the left-slice category :
	\[\left[ \begin{tikzcd}[scale cd = 0.6] \EucCat_{\text{mono}} \arrow[dr, bend right=20, "j_2"] & & \AffineCat_{\text{mono}} \arrow[dl, bend left=20, "\pi"'] \\ & \VectCat_{\text{mono}} \end{tikzcd} \right]
	\begin{tikzcd} {} \arrow[rr, mapsto] && {} \end{tikzcd}
	\left[ \begin{tikzcd}[scale cd = 0.6] \base \arrow[dr, bend right=20, "j_2"] & & \main \arrow[dl, bend left=20, "\pi"'] \\ & \inter \end{tikzcd} \right] =
	\Comma{j_2}{\pi} \]
	\item An affine map $T : \real^d \to V$ with an injective linear part corresponds to an object of the left-slice category :
	\[ \left[ \begin{tikzcd}[scale cd = 0.6] \EucCat_{\text{mono}} \arrow[dr, "j_1 j_2"', bend right=10] & & \AffineCat_{\text{mono}} \arrow[dl, "=", bend left=10] \\ 
		& \AffineCat_{\text{mono}} \end{tikzcd} \right]
	\begin{tikzcd} {} \arrow[rr, mapsto] && {} \end{tikzcd}
	\left[ \begin{tikzcd}[scale cd = 0.6] \base \arrow[dr, "j_1 j_2"', bend right=10] & & \main \arrow[dl, "=", bend left=10] \\ 
		& \main \end{tikzcd} \right] =
	\Comma{j_1 j_2}{ \main } . \]
	\item An affine map projects into a linear map. This is borne by the following functor between comma categories
	\begin{equation} \label{eqn:def:pi_star}
		\begin{tikzcd}
			\base \arrow[d, "="'] \arrow[r, "j_1 j_2"] & \main \arrow[d, "\pi"'] & \main \arrow[l, "="'] \arrow[d, "="] \\
			\base \arrow[r, "j_2"] & \inter & \main \arrow[l, "\pi"'] 
		\end{tikzcd} \imply 
		\begin{tikzcd} \Comma{j_1 j_2}{ \main } \arrow[Shobuj, "\pi_*", d] \\ \Comma{j_2}{\pi} \end{tikzcd}
	\end{equation}
	The two rows on the diagram on the left correspond to the two comma categories indicated on the diagram on the right. The induced functor $\pi_*$ is explained later in Lemma \ref{lem:oh9d}.
	\item Finally each of the objects in $\Comma{j_1 j_2}{ \main }$ are also set-maps. Consider the commuting diagram below on the left 
	\begin{equation} \label{eqn:def:N_I_B}
		\begin{tikzcd}
			\base \arrow[d, "="'] \arrow[r, "j_1 j_2"] & \main \arrow[d, "\gamma"'] & \main \arrow[l, "="'] \arrow[d, "="] \\
			\base \arrow[r, "\gamma j_1 j_2"] & \SetCat & \main \arrow[l, "\gamma"'] 
		\end{tikzcd} \imply 
		\begin{tikzcd}
			\Comma{j_1 j_2}{ \main } \arrow[Shobuj, dashed, drrr, "\Null_{\inter, \base}"] \arrow[d] \\
			\Comma{\gamma j_1 j_2}{ \gamma } \arrow[rrr, "\Nullity_{j_1 j_2, \main}"'] &&& \Nullity
		\end{tikzcd}
	\end{equation}
	The top and bottom rows correspond to two comma categories. Such a commutation leads to a functor between the comma categories, displayed as the unlabeled arrow in the diagram on the right. The bottom horizontal arrow is created using the construction in \eqref{eqn:ArrSet:1}. The composition of these two functors leads to a nullity structure on $\Comma{j_1 j_2}{ \main }$.
	\item Diagram \eqref{eqn:def:N_I_B} thus states that every affine embedding $T:\real^d \to V$ is related functorially to a nullity structure on $V$.
\end{enumerate}

We now have all the ingredients for re-inventing the construction of prevalence as presented in \eqref{eqn:prevalence:6} and Figure \ref{fig:prevalence:1}. The construction is done using a powerful tool from Category theory.

\paragraph{Kan extensions} Kan extensions \citep[e.g.]{perrone2022kan, street2004categorical, Riehl_homotopy_2014} are universal constructions which generalize the practice of taking partial minima or maxima, in a functorial manner.Consider the following arrangement :
\begin{equation} \label{eqn:dpp3k}
	\begin{tikzcd}
		X \arrow{r}{F} \arrow{d}[swap]{K} & E \\ D
	\end{tikzcd}
\end{equation}
of functors and categories. A \emph{left Kan extension} or \emph{right envelope} of $F$ along $K$ is a functor $\psi : D\to E$ along with a minimum natural transformation $\eta : F\Rightarrow \psi \circ K$. With a slight departure from usual convention, we denote this functor $\psi$ as $\REnv{K}{F}$. This pair $\paran{ \REnv{K}{F}, \eta }$  is also minimum / universal in the sense that for every other functor $H:D\to E$ along with a natural transformation $\gamma : F \Rightarrow H\circ K$, there is a natural transformation $\tilde\gamma : \REnv{K}{F} \Rightarrow H$ s.t. $\gamma = \paran{ \tilde\gamma \star \Id_{K} } \circ \eta$.  This is shown in the diagram below.
\[ \mathcal{L} := \REnv{K}{F}, \quad 
\begin{tikzcd}
	& & & E \\
	E & X \arrow[dashed, bend left = 10]{urr}[name=x1]{} \arrow[dashed]{drr}[name=x2]{} \arrow{l}[name=F]{F} \arrow{r}{K} & D \arrow{dr}[name=H]{H} \arrow{ur}[swap, name=L]{ \mathcal{L} } & \\
	& & & E 
	\arrow[shorten <=2pt, shorten >=3pt, Rightarrow, to path={(F) to[out=90,in=180] (x1)} ]{ }
	\arrow[shorten <=2pt, shorten >=3pt, Rightarrow, to path={(F) to[out=-90,in=225] (x2)} ]{ }
	\arrow[shorten <=1pt, shorten >=1pt, Rightarrow, to path={(L) to[out=-45,in=45] (H)} ]{ }
\end{tikzcd}\]
One can similarly define a \emph{right Kan extension} or \emph{left-envelope} of $F$ along $K$. It is a functor $\LEnv{K}{F} : D\to E$ along with a natural transformation $\epsilon : \LEnv{K}{F} \circ K \Rightarrow F$. Moreover, this pair $\paran{\LEnv{K}{F}, \epsilon}$ is maximum / universal in the sense that for every other functor $H:D\to E$ along with a natural transformation $\gamma : H\circ K \Rightarrow F$, there is a natural transformation $\tilde{\gamma} : H \Rightarrow \LEnv{K}{H}$ such that $\gamma = \epsilon\circ \paran{ \tilde{K} \star \Id_K }$. This is shown in the diagram below.
\[ \mathcal{R} := \LEnv{K}{F}, \quad 
\begin{tikzcd}
	& & & E \\
	E & X \arrow[dashed, bend left = 10]{urr}[name=x1]{} \arrow[dashed]{drr}[name=x2]{} \arrow{l}[name=F]{F} \arrow{r}{K} & D \arrow{dr}[name=H]{H} \arrow{ur}[swap, name=L]{ \mathcal{R} } & \\
	& & & E 
	\arrow[shorten <=2pt, shorten >=3pt, Rightarrow, to path={(x1) to[out=135,in=90] (F)} ]{ }
	\arrow[shorten <=2pt, shorten >=3pt, Rightarrow, to path={(x2) to[out=225,in=-90] (F)} ]{ }
	\arrow[shorten <=1pt, shorten >=1pt, Rightarrow, to path={(H) to[out=45,in=-45] (L)} ]{ }
\end{tikzcd}\]
The act of finding limits or colimits is analogous to finding the minimum or maximum under this constraint. Many constructions in mathematics which are analogous to constrained optimizations, can be succinctly expressed in the language of Kan extensions.  

\paragraph{Max-min optimization} Using the language of Kan extensions, we make the following Kan extensions using the functors $\pi_*$ and $\Null_{\inter, \base}$ constructed above :
\begin{equation} \label{eqn:nullity_construct}
	\begin{tikzcd} [ scale cd = 1.5, column sep = large]
		\Comma{j_1 j_2}{ \main } \arrow[dd, "\pi_*"', Holud ] \arrow[Akashi]{rrrr}{\Null_{\inter, \base}}[swap, name=n1]{} \arrow[Holud, dashed, bend right=10]{ddrrrr}[name=n2]{} &&&& \Nullity \\
		\\
		\Comma{j_2}{\pi} \arrow[dd, "\Forget_2"'] \arrow[Holud]{rrrr}[name=n3, pos=0.2]{ \text{probed}-\Null}[name=n3, swap]{ \LEnv{ \pi_* }{ \Null_{\inter, \base} } } \arrow[Shobuj, dashed, bend right=10]{ddrrrr}[name=n4]{} &&&& \Nullity \\
		\\
		\main \arrow[Shobuj]{rrrr}[pos=0.2]{\Null}[swap]{ \REnv{ \Forget_2 }{ \LEnv{ \pi_* }{ \Null_{\inter, \base} } } } &&&& \Nullity
		\arrow[shorten <=1pt, shorten >=1pt, Rightarrow, to path={(n3) to[in=90,out=-90] (n4)} ]{ }
		\arrow[shorten <=1pt, shorten >=1pt, Rightarrow, to path={(n2) to[in=-90,out=45] (n1)} ]{ }
	\end{tikzcd}
\end{equation}
The middle arrow in yellow, achieves the $\main$-invariance of nullity by virtue of being a left-envelope (i.e. right Kan extension). However it is not a nullity structure on $\main$ itself, but on morphisms sourced from $\base$-objects via $\inter$. Borrowing the terminology from \cite{HSYprevalence1992}, we call such a morphism a \emph{probe}. Thus this notion of nullity is tied to a choice of a probe object, and we call this a \textit{probed notion of nullity}. The lowermost arrow in green, represents the construction of nullity for the main category $\main$. By virtue of being a right-envelope (i.e. left Kan extension), it is the union of all probed nullities. In other words it is the minimal nullity structure that contains the nullity structure produced by all the probes. Figure \ref{fig:prevalence:2} retraces the steps in this abstract categorical construction to relate it to the familiar construction of prevalence.

The extension of nullity was done solely using collection of morphisms available within the arrow category $\Comma{j_2}{\pi}$. Thus although intuitively we interpret nullity as a pre-existing structure and nullity-preservation as a derived property, this categorical construction assumes the morphisms as a primitive notion, and optimizes the nullity structure that would be preserved. This is in principle similar to several approaches to categorification, such as quasi-Borel spaces \cite{HeunenEtAl2017cnvnt},
Chen spaces \cite{Chen1977iterated} and diffeological spaces \cite{Stacey2008smooth, BaezHoffnung2011cnvnnt}.

\begin{figure}[!t]
	\centering
	\begin{tikzpicture} \node[draw, inner sep=5pt, draw=ChhaiD, line width=2pt] (box){ 
			\begin{tikzcd} [row sep = large, scale cd = 0.9]
				\chhaiB{ \braces{ \begin{array}{c} \mbox{Linear embeddings} \\ A : \real^d \to V \\ \mbox{and affine embeddings }  \\ T : \real^d \to V \\ \mbox{s.t. } \pi(T) = A  \end{array} } } \arrow[ChhaiB, dotted, r] & \Comma{j_1 j_2}{ \main } \arrow[d, "\pi_*"', Holud ] \arrow[Akashi, dashed]{rrr}{\Null_{\inter, \base}}[swap, name=n1]{} \arrow[Holud, dashed, bend right=10]{drrr}[name=n2]{} &&& \Nullity & \chhaiB{ \braces{ \begin{array}{c} \mbox{Nullity induced} \\ \mbox{by set-map } T \end{array} } } \arrow[ChhaiB, dotted, l] \\
				\chhaiB{ \braces{ \begin{array}{c} \mbox{Linear embeddings} \\ A : \real^d \to V \end{array} } } \arrow[ChhaiB, dotted, r] & \Comma{j_2}{\pi} \arrow[d, Shobuj, "\Forget_2"'] \arrow[Holud]{rrr}[name=n3, pos=0.2]{\Null}[name=n3, swap]{ \LEnv{ \pi_* }{ \Null_{\inter, \base} } } \arrow[Shobuj, dashed, bend right=10]{drrr}[name=n4]{} &&& \Nullity & \chhaiB{ \Null(A; V) } \arrow[ChhaiB, dotted, l] \\
				\chhaiB{ \braces{ \mbox{Vector field } V } } \arrow[ChhaiB, dotted, r] & \main \arrow[Shobuj]{rrr}[pos=0.3]{\Null}[swap, name=n7]{ \REnv{ !V }{ \LEnv{ \pi_* }{ \Null_{\inter, \base} } } } &&& \Nullity & \chhaiB{ \Null(V) } \arrow[ChhaiB, dotted, l] \\
				\\
				{} \arrow[phantom, ChhaiB]{rr}[name=n5]{\begin{array}{c} \mbox{Intersection / min optimization} \\ \mbox{over all T s.t. } \pi(T)=A \end{array}} && {} & {} \arrow[phantom, ChhaiB]{rr}[name=n6]{\begin{array}{c} \mbox{Union / max optimization} \\ \mbox{over all embeddings } A \end{array}} && {}
				\arrow[shorten <=1pt, shorten >=1pt, Rightarrow, to path={(n3) to[in=90,out=-90] (n4)} ]{ }
				\arrow[shorten <=1pt, shorten >=1pt, Rightarrow, to path={(n2) to[in=-90,out=45] (n1)} ]{ }
				\arrow[shorten <=1pt, shorten >=1pt, dotted, ChhaiB, to path={(n5) to[in=225,out=45] (n3)} ]{ }
				\arrow[shorten <=1pt, shorten >=1pt, dotted, ChhaiB, to path={(n6) to[in=-45,out=135] (n7)} ]{ }
			\end{tikzcd}
		}; \end{tikzpicture}
	\caption{Categorical explanation of prevalence. Figure \ref{fig:prevalence:1} presented a flowchart outlining the construction the notion of prevalence \eqref{eqn:prevalence:6}. The categorical construction in \eqref{eqn:nullity_construct} has been superimposed om that flowchart. The various collections in Figure \ref{fig:prevalence:1} can now be related to various categories. The correspondences between the collections turn out to be functors. The construction of prevalence is done using a successive max / min optimizations. These optimization steps are shown to correspond to Kan extensions. This categorical interpretation completely dissociates the construction of [revalence] from the context to vectors spaces and affine maps. It therefore becomes applicable to any other context, that satisfies certain structural or categorical assumptions (\ref{A:1}--\ref{A:3}). }
	\label{fig:prevalence:2} 
\end{figure}

This completes a categorical redefinition of prevalence and shy sets, in an abstract categorical setting. The three ingredients are -- an interpretation $\gamma$ of objects from the main category $\main$ as sets (Assumption \ref{A:1}); a pre-existing notion of nullity on the base category $\base$ (Assumption \ref{A:2}); and finally the arrangement in \eqref{eqn:A:3_4} (Assumption \ref{A:3}). The terms of Assumption \ref{A:1} are fixed. A construction of nullity thus involves two independent choices of diagrams \eqref{eqn:null:4} and \eqref{eqn:A:3_4}. Figure \ref{fig:A:3_4} presents several instances of these assumptions which lead to other notions of nullity.

\begin{example} [Prevalence] \label{ex:preval}
	As declared before, prevalence is the special case of \eqref{eqn:nullity_construct} displayed in Figure \ref{fig:A:3_4:2}. The base notion of nullity which is used is given in Example \ref{ex:1}. Figure \ref{fig:prevalence:2} elaborates this connection.
\end{example}

Note that the notion of nullity is not tied to a choice of measure. It is ultimately a down-set of subsets. Multiple translation invariant measures have been proposed for infinite dimensional Banach spaces \cite[e.g.]{gill2013ordinary, Pantsulaia2008generators, vershik2007does, kakutani1950construction, baker1991lebesgue} whose null sets coincide with the null sets resulting from Example \ref{ex:preval}.  We next examine the mathematical consequences of the construction in \eqref{eqn:nullity_construct}.

\section{Main results} \label{sec:results}

The most trivial consequence of the categorical nature of our constructions is :

\begin{theorem} [Invariance of Nullity] \label{thm:1}
	Let Assumptions \ref{A:1}, \ref{A:2} and \ref{A:3} hold. Then the nullity created for $\main$ using the construction \eqref{eqn:nullity_construct} is invariant under the endomorphisms of $\main$.
\end{theorem}

The invariance follows from our interpretation in Definition \ref{def:null:4} of nullity as a functor, and the categorical structure of both $\main$ and $\Nullity$. The next important property to establish is uniqueness. Uniqueness can be established on the basis of the following desirable property :

\begin{definition} [Testability] \label{def:probe_resist}
	Let $\Null$ be a nullity construct according to Definition \ref{def:null:3} of an object $V$ of $\main$. This nullity is said to be testable if there is a probe object $\phi : j_2(b) \to \pi(V)$ in $\inter$ such that the push-forward of $\Null(b)$ under $\phi$ is a sub-collection of $\Null(V)$.
\end{definition}

Thus if a nullity construct is testable, a sub-collection of its null sets can be partially verified by null sets from a probe object sent from the base category. This is in fact the implicit motivation behind the notion of "prevalence".

\begin{theorem} [Nullity is unique] \label{thm:3}
	Let Assumptions \ref{A:1}, \ref{A:2} and \ref{A:3} hold. Let $\tilde{\Null}$ be a nullity construct for objects of $\main$,s satisfying the following two criterion
	\begin{enumerate} [(i)]
		\item $\tilde{\Null}$ is preserved under morphisms in $\main$;
		\item $\tilde{\Null}$ is testable.
	\end{enumerate}
	Then $\Null(V)$ constructed from \eqref{eqn:nullity_construct} is a sub-collection of $\tilde{\Null}(V)$.
\end{theorem}

Condition (i) of Theorem \ref{thm:3} is just the condition for functoriality. Condition (ii) can be interpreted as $\tilde{\Null}(V)$ containing the probed nullity structure of some probe. Theorem \ref{thm:3} is thus a direct interpretation of the successive Kan extensions involved in the construction of Nullity. The categorical construction provides the simplest or coarsest notion of nullity satisfying the conditions of Theorem \ref{thm:3}. The next property to be expected from \eqref{eqn:nullity_construct} is that the newly constructed nullity functor on $\main$ be an extension of the existing nullity on $\base$. There is however an a basic obstruction to this happening. This will come to light once we examine some simpler instances of Assumption \ref{A:3}.

\paragraph{Nullity from a single functor} Consider the particular case of Assumption \ref{A:3} displayed in Figure \ref{fig:A:3_4:7}. In that case \eqref{eqn:nullity_construct} takes the form :
\begin{equation} \label{eqn:def:veeNull:1}
	\begin{tikzcd} [ row sep = large]
		\Comma{j_2}{ \main } \arrow[d, "="' ] \arrow{rrrr}{\Null_{j_2, \main}}[swap, name=n1]{} \arrow[dashed, bend right=10]{drrrr}[name=n2]{} &&&& \Nullity \\
		\Comma{j_2}{ \main } \arrow[d, "\Forget_2"'] \arrow{rrrr}[name=n3, pos=0.2]{ \Null_{j_2, \main}}[name=n3, swap]{} \arrow[dashed, bend right=10]{drrrr}[name=n4]{} &&&& \Nullity \\
		\main \arrow{rrrr}[pos=0.2]{\vee \Null^{\base}}[swap]{ \REnv{ \Forget_2 }{ \Null_{j_2, \main} } } &&&& \Nullity
		\arrow[shorten <=1pt, shorten >=1pt, Rightarrow, to path={(n3) to[in=90,out=-90] (n4)} ]{ }
		\arrow[shorten <=1pt, shorten >=1pt, Rightarrow, to path={(n2) to[in=-90,out=45] (n1)} ]{ }
	\end{tikzcd}
\end{equation}
This construct has the following pointwise definition --
\begin{equation} \label{eqn:def:veeNull:2}
	\begin{split}
		\vee \Null^{\base} (A) &= \bigcup \SetDef{ \Null_{j_2, \main}(\phi) }{ \phi : j_1(A') \to A } \\
		&= \bigcup \SetDef{ a\subseteq \gamma(A) }{ \exists \phi : j_2(A') \to A, \, (\gamma \phi)^{-1}(a) \in \Null(A') }
	\end{split}
\end{equation}
Thus just three ingredients -- a pre-existing nulliy structure $\Null^{(\base)}$ on $\base$; a functor $j_2 : \base \to \main$; and a set-functor $\gamma : \main\to \SetCat$ leads to a nullity structure for $\main$. We next examine an even simpler instance of \eqref{eqn:def:veeNull:1}, when all $\base, \inter, \main$ are the same.

\paragraph{Enriching nullity structure} Consider the particular case of Assumption \ref{A:3} displayed in Figure \ref{fig:A:3_4:3}. In that case \eqref{eqn:nullity_construct} as well as \eqref{eqn:def:veeNull:1} takes the form :
\begin{equation} \label{eqn:def:Bar_Null:1}
	\begin{tikzcd} [row sep = large]
		\ArrowCat{\base} \arrow[d, "="' ] \arrow{rrrr}{ \Null_{\base, \base} }[swap, name=n1]{} \arrow[dashed, bend right=10]{drrrr}[name=n2]{} &&&& \Nullity \\
		\ArrowCat{\base} \arrow[d, "\Forget_2"'] \arrow{rrrr}[name=n3, pos=0.2]{ \Null_{\base, \base} }[name=n3, swap]{ } \arrow[dashed, bend right=10]{drrrr}[name=n4]{} &&&& \Nullity \\
		\base \arrow{rrrr}[pos=0.2]{\overline{\Null}} &&&& \Nullity
		\arrow[shorten <=1pt, shorten >=1pt, Rightarrow, to path={(n3) to[in=90,out=-90] (n4)} ]{ }
		\arrow[shorten <=1pt, shorten >=1pt, Rightarrow, to path={(n2) to[in=-90,out=45] (n1)} ]{ }
	\end{tikzcd}
\end{equation}
Note that the induced map $\pi_*$ is just an identity between comma categories. As a result, the first right envelope is just the original functor. The new creation is the functor $\overline{\Null}$. Its action on morphisms of $\base$ can be formulated explicitly as
\begin{equation} \label{eqn:def:Bar_Null:2}
	\begin{split}
		\overline{\Null}(A) &= \bigcup \SetDef{ \Null_{\base, \base}(\phi) }{ \phi : A' \to A } \\
		&= \bigcup \SetDef{ a\subseteq \gamma(A) }{ \exists \phi : A' \to A, \mbox{ s.t. } (\gamma \phi)^{-1}(a) \in \Null(A') }
	\end{split}
\end{equation}
Note that $\overline{\Null}$ is a superset of $\Null$. Take any $a\in \Null(A)$, and take $\phi$ to be the identity $\Id_A : A\to A$. Then by \eqref{eqn:def:Bar_Null:2} $a$ lies in $\overline{\Null}(A)$ too. In general one cannot expect $\overline{\Null}$ to coincide with $\Null$.

\begin{definition} [Saturated nullity structure] \label{def:sat_null}
	Suppose Assumptions \ref{A:1} and \ref{A:2} hold. Then the nullity structure is said to be saturated if the functor $\overline{\Null}$ from \eqref{eqn:def:Bar_Null:1}, \eqref{eqn:def:Bar_Null:2} coincides with $\Null$.
\end{definition}

To prove our main result, we need an assumption based on the following inclusion functor :
\begin{equation} \label{eqn:null_ext:7}
	\begin{tikzcd}
		\base \arrow[r, "="] \arrow[d, "="] & \base \arrow[d, "j_2"] & \base \arrow[l, "="'] \arrow[d, "j_2"] \\
		\base \arrow[r, "j_2"'] & \inter & \inter \arrow[l, "="]
	\end{tikzcd} \imply 
	\begin{tikzcd} \ArrowCat{\base} \arrow[d, "\incl_3", Shobuj] \\ \Comma{j_2}{\inter} \end{tikzcd}
\end{equation}
The inclusion $\incl_3$ formalizes the intuition that arrows in $\base$ are also arrows in $\inter$ from a $\base$ object to an $\inter$-object. We assume the following about $\incl_3$ :

\begin{Assumption} \label{A:4} 
	The functor $\incl_3$ defined in \eqref{eqn:null_ext:7} has a post right adjoint $\incl_3^{R*} : \Comma{j_2}{\inter} \to \ArrowCat{\base}$.
\end{Assumption}

The notion of a \textit{post-right adjoint}, exlplained in Section \ref{sec:appendix:Kan}, is a generalization of being a right inverse. In the instance of Figure \ref{fig:A:3_4:2}, $\incl_3^{R*}$ is simply the functor which assigns every linear embedding $A : \real^d \to V$ the linear map $A : \real^d \to \ran(A)$ between finite dimensional vector spaces. Thus in the case of prevalence, $\incl_3^{R*}$ is actually a right inverse of $\incl_3$. We will find use for one more assumption on the arrangement in \eqref{eqn:A:3_4}. A functor $F:X\to Y$ is said to be \emph{full} if for every $x, x'\in X$ , the action of $F$ is surjective on $\Hom_{Y}(Fx; Fx')$. In other words, every morphism in $Y$ from $Fx$ to $Fx'$ is the image under $F$ of some morphism in $X$.

\begin{Assumption} \label{A:5}
	The functor $j_2$ is full.
\end{Assumption}

Assumption \ref{A:5} is trivially true for prevalence, as evident from Figure \ref{fig:A:3_4:4}.

\begin{theorem} [Nullity is an extension] \label{thm:null_ext}
	Let Assumptions \ref{A:1}, \ref{A:2}, \ref{A:3} and \ref{A:4} hold, and there is a nullity structure $\Null^{(\base)}$ on $\base$, and $\Null^{(\main)}$ is the extended notion of nullity according to the construction in Diagram \eqref{eqn:null_ext:7}. Then :
	\begin{enumerate} [(i)]
		\item There is an inclusion of nullity structures :
		\[ \Null^{(\base)}(x) \subseteq \overline{ \Null^{(\base)} } (x) \subseteq \Null^{(\main)}(x), \quad \forall x\in \base. \]
		\item Suppose Assumption \ref{A:5} also holds. Then $\overline{ \Null^{(\base)} }$ coincides with $\Null^{(\main)}(x)$.
		\item Further suppose that $\Null^{(\base)}$ is saturated. Then
		\[ \Null^{(\base)}(x) = \overline{ \Null^{(\base)} } (x) = \Null^{(\main)}(x), \quad \forall x\in \base. \]
	\end{enumerate}	 
\end{theorem}

\paragraph{Remark} Claim (i) of Theorem \ref{thm:3} reveals that if the pre-existing nullity structure on $\base$ is not saturated, then the categorical extension to $\main$ will not be an exact extension. 

\paragraph{Remark} Both Assumptions \ref{A:4} and \ref{A:5} are structural assumptions on how the functors $j_1$ and $j_2$ transform morphisms within $\inter$ and $\main$. These assumptions are not required for the construction in Diagram \eqref{eqn:nullity_construct}, but guarantee certain desirable properties.

The proof of Theorem \ref{thm:null_ext} depends on realizing certain functorial relations between comma categories. The commutative diagram on the left :
\begin{equation} \label{eqn:null_ext:1}
	\begin{tikzcd}
		\base \arrow[r, "="] \arrow[d, "="] & \base \arrow[d, "j_2"] & \base \arrow[l, "="'] \arrow[d, "j_1 j_2"] \\
		\base \arrow[r, "j_2"'] & \inter & \main \arrow[l, "\pi"]
	\end{tikzcd} \imply 
	\begin{tikzcd} \ArrowCat{\base} \arrow[d, "\incl_1", Shobuj] \\ \Comma{j_2}{\pi} \end{tikzcd}
\end{equation}
leads to an induced functor $\incl_1$ between the comma categories represented by the two horizontal rows. One similarly gets an induced functor $\incl_2$ as shown below :
\begin{equation} \label{eqn:null_ext:2}
	\begin{tikzcd}
		\base \arrow[r, "="] \arrow[d, "="] & \base \arrow[d, "j_1 j_2"] & \base \arrow[l, "="'] \arrow[d, "j_1 j_2"] \\
		\base \arrow[r, "j_1 j_2"'] & \main & \main \arrow[l, "="]
	\end{tikzcd} \imply 
	\begin{tikzcd} \ArrowCat{\base} \arrow[d, "\incl_2", Shobuj] \\ \Comma{j_1 j_2}{ \main } \end{tikzcd}
\end{equation}
The proof can be summarized in the following large commutation diagram that involves these functors.
\begin{equation} \label{eqn:null_ext:3}
	\begin{tikzcd} [scale cd = 1.3]
		\Comma{\gamma}{ \gamma } \arrow[rr] && \Comma{\gamma j_1 j_2}{ \gamma } \arrow[drrrr, bend left=10, "\Nullity_{j_1 j_2, \main}"] \\
		&& \Comma{j_1 j_2}{ \main } \arrow[dd, "\pi_*"' ] \arrow[Holud]{rrrr}{\Null_{\inter, \base}}[swap]{} \arrow[u] &&&& \Nullity \\
		\ArrowCat{\base} \arrow[uu] \arrow[dd, "="'] \arrow[urr, "\incl_2", Holud] \arrow[urrrrrr, bend right = 20, dotted, "\Null_{\base, \base}"', pos=0.8, Holud] \\
		&& \Comma{j_2}{\pi} \arrow[dd, "\Forget_2"'] \arrow[dashed, Akashi]{rrrr}{ \LEnv{ \pi_* }{ \Null_{\inter, \base} } }  &&&& \Nullity \\
		\ArrowCat{\base} \arrow[dd, "\Forget_2^{\base}"'] \arrow[urr, "\incl_1", Akashi] \arrow[urrrrrr, bend right = 20, dotted, "\Null_{\base, \base}"', pos=0.8, Akashi] \\
		&& \main \arrow[swap, pos=0.2, dashed, red]{rrrr}[name=n2]{\Null^{(\main)}} \arrow[dashed, red]{rrrr}{ \REnv{ \Forget_2 }{ \LEnv{ \pi_* }{ \Null_{\inter, \base} } } } &&&& \Nullity \\
		\base \arrow[urr, "j_1 j_2", red] \arrow[dotted, red]{rrrrrr}[pos=0.8, swap, name=n1]{ \overline{\Null^{(\base)}}} && &&&& \Nullity
		\arrow[shorten <=1pt, shorten >=1pt, Rightarrow, to path={(n1) to[out=90,in=-90] (n2)} ]{ }
	\end{tikzcd}
\end{equation}
The dashed arrows represent the Kan extensions displayed in \eqref{eqn:nullity_construct}, while the dotted arrows represent Kan extension displayed in \eqref{eqn:def:Bar_Null:1}. The main message of this diagram is the commutation between the corresponding pairs of Kan extensions. The lowermost commutation is precisely the extension claimed in Theorem \ref{thm:null_ext}. Thus the proof requires a verification of \eqref{eqn:null_ext:3}. See Section \ref{sec:proof:null_ext} for the details of the proof. 

We next look at more applications of Diagram \eqref{eqn:nullity_construct} and Theorems \ref{thm:1}, \ref{thm:null_ext} and \ref{thm:3}.

\section{Examples} \label{sec:examples}

In the following examples we use $\JetEucZ{k}$ to denote the category in which objects are Euclidean spaces, and the morphisms between two spaces $\real^m$ and $\real^n$ are $C^{k}$ maps. Such morphisms will be called $\JetEucZ{k}$-maps. Let $\Poly{k}$ denote the category in which objects are Euclidean spaces and morphisms are polynomials of degree $k$. The composition rule for morphisms is the usual composition rule for polynomials, modulo the symmetric functions of degree $k+1$.

\begin{example} [Jets 1] \label{ex:JetEuc:1}
	Figure \ref{fig:A:3_4:4} presents an arrangement of $\JetEucZ{1}$ and $\Poly{1}$ in the pattern of \eqref{eqn:A:3_4}. Here $j_1, j_2$ are inclusion functors, while $\pi$ is the functor which assigns to every $C^k$ map the collection of its first $k$ derivatives. is the identity on objects, and transforms any map $f:\real^m \to \real^n$ into the polynomial obtained from the order $k$ Taylor expansion of $f$. Thus the construction \eqref{eqn:nullity_construct} applies and results in a notion of nullity for the category $\JetEucZ{1}$ : A subset $S$ of $\real^n$ is null iff there is an affine map $A : \real^m \to \real^n$ such that for every $C^1$ map $T : \real^m \to \real^n$ whose 1-jet coincides with $A$, one has $\Leb_{\real^d} \paran{ T^{-1} (S) } = 0$.
\end{example}

The next example shows how the principle in Diagram \eqref{eqn:nullity_construct} can be used in an repeatedly to obtain nullity constructs for $\JetEucZ{k}$ inductively on $k$.

\begin{example} [Jets 2] \label{ex:JetEuc:k}
	Continuing the discussion in Example \ref{ex:JetEuc:1},  $\JetEucZ{k}$, $\JetEucZ{k+1}$ and $\Poly{k+1}$ can be arranged as shown in Figure \ref{fig:A:3_4:5} in the pattern of \eqref{eqn:A:3_4}. Thus the construction \eqref{eqn:nullity_construct} applies and results in a notion of nullity for the category $\JetEucZ{k+1}$ : A subset $S$ of $\real^d$ is null iff there is a degree-$k$ polynomial $A : \real^m \to \real^n$ such that for every $C^{k+1}$-map $T$ whose $k$-jet is $A$, the set $T^{-1} (S)$ is null in $\JetEucZ{k}$.
\end{example}

\begin{example} [Jets 3] \label{ex:JetEuc:k:quotient}
	 The nullity structure derived above in Example \ref{ex:JetEuc:k} is invariant under $C^k$-smooth maps. Let us call two morphisms $\phi, \phi':X\to Y$ in $\JetEucZ{k}$ equivalent if they are convertible by an invertible linear transformation.  We can construct a quotient category $\JetEucZ{k} / _{\sim}$ with the same set of objects, and in which the morphisms are equivalences classes. Then the above nullity construct is obviously a valid construct on $\JetEucZ{k} / _{\sim}$.
\end{example} 

All the categories $\EucCat_{\text{mono}}$, $\EucCat$, $\Poly{k}$ and $\JetEucZ{k}$ for various $k$ have the same set of objects. The difference is in the morphisms allowed. The following chain of inclusions is easily observed :
\[ \EucCat_{\text{mono}} \subset \EucCat \subset \Poly{k} \subset \JetEucZ{k} .\]
As a result the nullity constructs in Examples \ref{ex:preval}, \ref{ex:JetEuc:1} and \ref{ex:JetEuc:k} are successively richer, in the sense that they remain invariant under a richer set of morphisms. In the next examples, we use $\ManCat{k}_{pt}$ to denote the category of pointed manifolds. Thus objects are pairs $(x,M)$ where $x$ is a point on a manifold $M$ and a morphism $(x,M) \xrightarrow{\phi} (x', M')$ is a $C^k$-differentiable map $\phi$ such that $\phi(X) = x'$. 

\begin{example} [Pointed manifolds] \label{ex:man_pt:k}
	The categories $\JetEucZ{k}$, $\ManCat{k}_{pt}$ can be arranged as shown in Figure \ref{fig:A:3_4:6} in the pattern of \eqref{eqn:A:3_4}. Thus the construction \eqref{eqn:nullity_construct} applies and results in a notion of nullity for the category $\ManCat{k}_{pt}$ : A subset $S$ of a pointed $C^k$ manifold $(x,M)$ is null iff there is a $C^{k}$-map $T:\real^d \to M$ such that for every $C^{k}$-map $T:\real^d \to M$ mapping the origin into $x$, whose $k$-jet at $x$ is $A$, the set $T^{-1} (S)$ is null in $\JetEucZ{k}$.
\end{example}

Note that the Jet of a function at a point on a manifold does not have a unique representation. The representation depends on the choice of coordinates. However all of them are related by a linear change of variables. This is the reason for choosing the category $\JetEucZ{k} / _{\sim}$.
In the next example we shall see how any such notion of nullity for pointed manifolds can be extended to ordinary manifolds. This is where the principle outlined in \eqref{eqn:def:veeNull:1} will be employed.

\begin{example} [Manifolds] \label{ex:man:k}
	Continuing the discussion in Example \ref{ex:man_pt:k}, the categories $\ManCat{k}$, $\ManCat{k}_{pt}$ can be arranged as shown in Figure \ref{fig:A:3_4:7} in the pattern of \eqref{eqn:A:3_4}. Thus the construction \eqref{eqn:nullity_construct} applies and results in a notion of nullity for the category $\ManCat{k}$ : A subset $S$ of a $C^k$ manifold $M$ is null iff there is an $x\in M$ such that $S$ is null in $(x,M)$.
\end{example}

This completes the presentation of our examples.

There are many commonly used spaces lacking a notion of "full" or "null". Assumptions \ref{A:1}, \ref{A:2}, \ref{A:3} and \ref{A:4} provide a means of filling these gaps. This task requires a pre-existing notion of nullity, as captured in \eqref{eqn:null:4}, and a link as in \eqref{eqn:A:3_4} to the target category. As we have argued before, the notion of prevalence is a special case of the construction in \eqref{eqn:nullity_construct}. Although the former was done using classical Analysis, relying on a categorical approach offers many advantages. Firstly it provides a neat organization and separation of the various spaces and relations. Secondly the invariance of the resultant nullity structure with respect to morphisms in the category are automatically guaranteed by functoriality. Finally the categorical approach distills out the most essential features of the construction. For example in the case of Prevalence, properties such as reflexivity, or norm and inner product structure was non-essential.

\begin{question} \label{Q:1}
	Suppose $\calC$ is a sub-category of $\overline{\calC}$ with full object set, and the former has a nullity structure $\Null : \calC \to \Nullity$. Then how can one characterize the maximal sub-category of $\overline{\calC}$ that preserves the same nullity structure ?
\end{question}

In many practical scenarios a transformation presents itself as a morphism in a category $\overline{\calC}$ which is richer than a categorical structure $\calC$. Then the challenge is to prove that the morphism preserves the notion of prevalence.
Once again, the categorical approach presented seems to be the most promising path to answering these questions.
\section*{Statements and Declarations} The author has no financial or non-financial interests conflicting with the preparation or publication of this article.

This work was not supported by any source of funding.

\section{Appendix} \label{sec:proofs}

\subsection{Functors induced between comma categories}

\begin{lemma} \label{lem:oh9d}
	\cite[Prop 6]{Das2024slice} Consider the arrangement of categories $\calA, \calB, \calC, \calD, \calE$
	\begin{equation} \label{eqn:gpd30}
		\begin{tikzcd}
			\calA \arrow{d}{I} \arrow[Holud]{r}{F} & \calB \arrow{d}[swap]{J} & \calC \arrow[Holud]{l}[swap]{G} \arrow{d}{K} \\
			\calA' \arrow[Akashi]{r}[swap]{F'} & \calB' & \calC' \arrow[Akashi]{l}{G'}
		\end{tikzcd}
	\end{equation}
	Then there is an induced functor between comma categories
	\begin{equation} \label{eqn:odl9}
		\Psi : \Comma{F}{G} \to \Comma{F'}{G'}, 
	\end{equation}
	Moreover, the following commutation holds with the marginal functors :
	\begin{equation} \label{eqn:dgbw}
		\begin{tikzcd} [column sep = large]
			\calA \arrow[d, "I"'] & \Comma{F}{G} \arrow[l, "\Forget_1"'] \arrow[r, "\Forget_2"] \arrow[d, "\Psi"] & \calC \arrow[d, "K"] \\
			\calA' & \Comma{F'}{G'} \arrow[l, "\Forget_1"] \arrow[r, "\Forget_2"'] & \calC'  
		\end{tikzcd}
	\end{equation}
\end{lemma}

Lemma \ref{lem:oh9d} has several applications. The first is an induced functor between arrow categories, created by a functor :
\begin{equation} \label{eqn:def:induce_arrow}
	\begin{tikzcd} \calX \arrow{d}[swap]{F} \\ \calY \end{tikzcd}
	\imply
	\begin{tikzcd} \ArrowCat{\calX} \arrow{d}{F_*} \\ \ArrowCat{\calY} \end{tikzcd}
\end{equation}
This is a result of a direct application of Lemma \ref{lem:oh9d} to the following special case of \eqref{eqn:gpd30} :
\[\begin{tikzcd}
	\calX \arrow{d}{F} \arrow{r}{=} & \calX \arrow{d}[swap]{F} & \calX \arrow{l}[swap]{=} \arrow{d}{F} \\
	\calY \arrow{r}[swap]{=} & \calY & \calY \arrow{l}{=}
\end{tikzcd}\]
Another important application is
\begin{equation} \label{eqn:sdo3h7}
	\begin{tikzcd} \calX \arrow[dr, "F"] & & \calY \arrow[dl, "G"'] \\ & \calZ \end{tikzcd} + 
	\begin{tikzcd} \calZ \arrow[d, "H"] \\ \calZ' \end{tikzcd} \imply 
	\begin{tikzcd}
		\calX \arrow{d}{=} \arrow{r}{F} & \calZ \arrow{d}[swap]{H} & \calY \arrow{l}[swap]{G} \arrow{d}{=} \\
		\calX \arrow{r}[swap]{HF} & \calZ & \calY \arrow{l}{HG}
	\end{tikzcd} \imply 
	\begin{tikzcd} \Comma{F}{G} \arrow[d, "H_*"] \\ \Comma{HF}{HG} \end{tikzcd}
\end{equation}
%

\subsection{Limits and co-limits} \label{sec:appendix:prepost_adj}

Suppose $J$ and $\calX$ are two categories. Then for any object $x\in ob(\calX)$ there is a constant functor that sends every object of $J$ into $x$, and every morphism into $\Id_x$. This functor is denoted as $x:J\to \calX$. Given a functor $F:J\to \calX$, a \emph{co-cone} for $F$ is a natural transformation $\eta : F \Rightarrow x$ from $F$ to a constant functor $x$. The object $x$ is called the \emph{tip} of the co-cone. A \emph{colimit} of $F$ is a universal co-cone, that means it is a co-cone $\bar{\eta} : F \Rightarrow \bar{x}$ such that for any other co-cone $\eta : F \Rightarrow x$, there is a unique $\phi : \bar{x} \to x$ such that $\eta = \phi \circ \bar{\eta}$. The tip $\bar{x}$ of the universal co-cone is also called the colimit of the functor $F$. Colimits are closely related to their dual notion : \emph{limits}.

Colimits and limits are a sweeping generalization of various fundamental constructs in mathematics, such as limits, convergence, min, max, union, intersection, inverse images, and gluing. The functor $J$ can interpreted as a pattern and $F$ as a diagram existing within $\calX$ in the shape of $J$. $J$ may be finite or infinite. For example a terminal object is thus the colimit of the identity functor $\Id_{\calX} : \calX \to \calX$. The colimit $\bar{x}$ provides a sense of completion or termination of the pattern. Colimits may not exist, but if they do, they are unique up to isomorphism.

The first technical result provides a sufficient condition under which pre-composition with a functor preserves limits and colimits.

\begin{lemma} \label{lem:preadj:1}
	\cite[Lem 4.2]{Das2023CatEntropy} Consider the following arrangement of functors and categories.
	\[\begin{tikzcd}
		X \arrow[r, "f"] & Y \arrow[r, "g"] & Z
	\end{tikzcd}\]
	\begin{enumerate} [(i)]
		\item If $f$ has a post-right adjoint $f^*$, and the colimits $\colim g$, $\colim gf$ and $\colim gff^*$ exist,  then $\colim gff^* = \colim fg = \colim g$.
		\item If $f$ has a pre-right adjoint $f^*$, and the limits $\lim g$, $\lim gf$ and $\lim gff^*$ exist,  then $\lim gff^* = \lim fg = \lim g$.
	\end{enumerate}
\end{lemma}

Before the proof, let us set the convention that a constant functor from a category $C$ with constant value $d$ in the codomain $D$, will be denoted as $\Delta_{C,d}$.

\begin{proof} We only prove the second claim, as the first follows from duality. By assumption there is a natural transformation
	\[ \eta : ff^* \Rightarrow \Id_{Y}. \]
	Let the limits of the functors $g$, $gf$, $gff^*$ be $c_{g}$, $c_{gf}$ and $c_{gff^*}$ respectively. Thus they create the limiting cones
	\[ \eta^{(g)} : \Delta_{Y, c_{g}} \Rightarrow g, \; \eta^{(gf)} : \Delta_{X, c_{gf}} \Rightarrow gf, \; \eta^{(gff^*)} : \Delta_{Y, c_{gff^*}} \Rightarrow gff^*, \; \]
	Note that $\Delta_{Y, c_{g}} \circ f = \Delta_{X, c_{g}}$ and $\Delta_{X, c_{gf}} \circ f^* = \Delta_{Y, c_{gf}}$. Then we have
	\begin{equation} \label{eqn:preadj:1:1}
		\begin{tikzcd}
			\Delta_{X, c_{g}} \arrow[Rightarrow]{dr}[swap]{ !\phi } \arrow[Rightarrow]{rr}{ \eta^{(g)} \Id_{f} } && gf \\
			& \Delta_{X, c_{gf}} \arrow[Rightarrow]{ur}[swap]{ \eta^{(gf)} }
		\end{tikzcd}, \quad 
		\begin{tikzcd}
			\Delta_{Y, c_{gf}} \arrow[Rightarrow]{dr}[swap]{ !\psi } \arrow[Rightarrow]{rr}{ \eta^{(gf)} \Id_{f^*} } && gff^* \\
			& \Delta_{Y, c_{gff^*}} \arrow[Rightarrow]{ur}[swap]{ \eta^{(gff^*)} }
		\end{tikzcd}
	\end{equation}
	The left and right diagrams show a commutation between functors and natural transformations, in the functor categories $\Functor{X}{Z}$ and $\Functor{Y}{Z}$ respectively. We also have :
	\begin{equation} \label{eqn:preadj:1:2}
		\begin{tikzcd} [column sep = large]
			\Delta_{Y, c_{g}} \arrow[Rightarrow]{d}[swap]{ \eta^{(g)} } & \Delta_{Y, c_{gff^*}} \arrow[Rightarrow]{l}[swap]{!\gamma} \arrow[Rightarrow]{d}{ \eta^{(gff^*)} } \\
			g & gff^* \arrow[Rightarrow]{l}{ \Id_{g} \eta }
		\end{tikzcd}
	\end{equation}
	By the universality of $\eta^{(g)}$ we also have the following commutation :
	\begin{equation} \label{eqn:preadj:1:3}
		\begin{tikzcd} [column sep = large]
			\Delta_{Y, c_{g}} \arrow[Rightarrow]{d}[swap]{ \eta^{(g)} } \arrow[Rightarrow]{dr}{ \eta^{(g)} ff^* }  \\
			g & gff^* \arrow[Rightarrow]{l}{ \Id_{g} \eta }
		\end{tikzcd}
	\end{equation}
	Next, composing the left commuting diagram of \eqref{eqn:preadj:1:1} with $f^*$ gives
	\begin{equation} \label{eqn:preadj:1:4}
		\begin{tikzcd}
			\Delta_{Y, c_{g}} \arrow[Rightarrow]{dr}[swap]{ !\phi f^*} \arrow[Rightarrow]{rr}{ \eta^{(g)} f f^* } && gff^* \\
			& \Delta_{Y, c_{gf}} \arrow[Rightarrow]{ur}[swap]{ \eta^{(gf)} f^* }
		\end{tikzcd}
	\end{equation}
	The commuting diagrams \eqref{eqn:preadj:1:1}, \eqref{eqn:preadj:1:2}, \eqref{eqn:preadj:1:3} and \eqref{eqn:preadj:1:4} combine to give :
	\[\begin{tikzcd} [column sep = large]
		\Delta_{Y, c_{g}} \arrow[Rightarrow]{ddr}[swap]{ \eta^{(g)}  ff^* } \arrow[Rightarrow, bend left=10, Akashi]{drr}{!\phi \Id_{f^*}} \arrow[Rightarrow]{dd}[swap]{ \eta^{(g)} } \\
		{} & \Delta_{Y, c_{gff^*}} \arrow[Rightarrow, Akashi]{ul}[swap]{!\gamma} \arrow[Rightarrow, Akashi]{d}{ \eta^{(gff^*)} } & \Delta_{Y, c_{gf}} \arrow[Rightarrow]{l}[swap]{!\psi} \arrow[Rightarrow, bend left=20, Akashi]{dl}{ \eta^{(gf)} \Id_{f^*} } \\
		g & gff^* \arrow[Rightarrow]{l}{ \Id_{g} \eta }
	\end{tikzcd}\]
	The diagram reveals that the limiting cone $\eta^{(gff^*)}$ factorizes as
	\[\eta^{(gff^*)} = \paran{ \eta^{(gf)} f^* } \paran{\phi f^*} \gamma. \]
	By the minimality of the limiting cone, this implies that both $\phi$ and $\gamma$ are isomorphisms. This implies that all the three limits $c_{g}$, $c_{gf}$ and $c_{gff^*}$ are the same.
\end{proof}

The next result shows how a left and right pre/post adjoint pair carries over to comma categories.

\begin{lemma} \label{lem:preadj:2}
	Consider the following arrangement of functors and categories.
	\[\begin{tikzcd} A \arrow[r, "a"] & B \arrow[r, "f"] & E \end{tikzcd}\]
	By Lemma \ref{lem:oh9d} there is an induced functor
	\[\begin{tikzcd} \Comma{fa}{E} \arrow[r, "a_*"] & \Comma{f}{E}  \end{tikzcd}\]
	Then $a^*$ above has a pre / post right adjoint if $a$ has a pre / post-right adjoint.
\end{lemma}

The pre- or post right adjoint $a^*$ leads to the following commuting diagrams that connect comma categories.
\[\begin{tikzcd}
	B \arrow[r, "fa"] \arrow[d, "a^*"] \arrow[dd, "aa^8", bend right=49] & E \arrow[d, "="] & E \arrow[l, "="'] \arrow[d, "="'] \\
	A \arrow[r, "a"] \arrow[d, "a"] & E \arrow[d, "="] & E \arrow[l, "="'] \arrow[d, "="'] \\
	B \arrow[r, "f"] & E & E \arrow[l, "="'] 
\end{tikzcd}\]
The claim of the lemma now follows from the observation that there is a natural transformation between $aa^*$ and $\Id_B$. \qed

\subsection{Kan extensions} \label{sec:appendix:Kan}

Kan extensions are functors, and it is often possible to determine their action on objects, as shown below :

\begin{lemma} \label{lem:Kan_pointwise}
	Consider the arrangement of \eqref{eqn:dpp3k}. Then 
	\begin{equation} \label{eqn:Kan_pointwise}
		\begin{split}
			\mbox{If } E \mbox{ is cocomplete, } \mbox{ then }& \REnv{K}{F}(d) = \colim \SetDef{Fx}{ Kx \to d } \\ 
			\mbox{If } E \mbox{ is complete, } \mbox{ then } & \LEnv{K}{F}(d) = \lim \SetDef{Fx}{ d \to Kx } \\ 
		\end{split}
	\end{equation}
\end{lemma}

The colimit and limit in \eqref{eqn:Kan_pointwise} are along slices of the object $d$ along $K$. This construction is known as the \emph{pointwise} definition of Kan extensions. Note that each slice, left or right, can be interpreted as a constraint on the objects of $K$.
If $E$ is a co-complete category, the left Kan extension always exists. Similarly if $E$ is a complete category, the right Kan extension always exists. 

We next present a series of lemmas that establish conditions under which Kan extensions are compatible with various pre and post compositions with functors. We begin by introducing a relaxed version of inverses.

\paragraph{Pre- and post- adjoints} Let $T : A \to B$ be a functor. Then $T^* : B \to A$ is said to be a \textit{pre-right adjoint} to $T$ if $T T^* \Rightarrow Id_B$ . Such a functor $T$ is said to be \textit{pre-right adjoint enabled} or in brief pre-r.a.e.. Similarly, $T^* : B \to A$ is said to be a \textit{post-right adjoint} to $T$ if $Id_B \Rightarrow T T^*$ . In that case $T$ will be called \textit{post-right adjoint} enabled or in brief post-r.a.e.. Note that a left (right) inverse is both a pre- and post adjoint. Similarly, a left (right) adjoint is both a pre- and post adjoint.

\begin{lemma} \label{lem:Kan_id:2}
	Consider the following arrangement of functors $a, b, f$ as shown below.
	\[\begin{tikzcd}
		A \arrow[r, "a"] & B \arrow[rr, "b"] \arrow[d, "f"'] && C \\
		& E \arrow[rr, "\LEnv{f}{b}"'] && C
	\end{tikzcd}\]
	\begin{enumerate} [(i)]
		\item If $C$ is co-complete and $a$ has a post-right adjoint, then $\REnv{f}{b} = \REnv{fa}{ba}$.
		\item If $C$ is complete and $a$ has a pre-right adjoint, then $\LEnv{f}{b} = \LEnv{fa}{ba}$.
	\end{enumerate}
\end{lemma}

\begin{proof} We only prove Claim (i) as Claim (ii) is analogous. Fix an object $e\in E$. We have the following commutation diagram
	\[\begin{tikzcd} [column sep = large]
		\Comma{fa}{e} \arrow[d, "\Forget_1^{(fa)}"'] \arrow[r, "a_*"] & \Comma{f}{e} \arrow[d, "\Forget_1^{(f)}"'] \arrow[dashed, bend left = 20, "", dr] \\
		A \arrow[r, "a"] & B \arrow[r, "b"] & C
	\end{tikzcd}\]
	The commutation square follows from Lemma \ref{lem:oh9d}, while the commutation triangle follows from construction. Since $C$ is co-complete, Lemma \ref{lem:Kan_pointwise}  applies and we have :
	\[ \REnv{fa}{ba}(e) = \colim ba \Forget_2^{(fa)} = \colim b \Forget_2^{(f)} a_*  = \colim b \Forget_2 = \REnv{f}{b}(e). \]
	The second equality follows from the commutative diagram, the third follows from Lemma \ref{lem:preadj:2} , and fourth and final equality follows again from Lemma \ref{lem:Kan_pointwise}. This completes the proof of Lemma \ref{lem:Kan_id:2}.
\end{proof}

\begin{lemma} \label{lem:Kan_id:3}
	Consider the following arrangement of functors $c, d, e$ as shown below.
	\[\begin{tikzcd}
		A \arrow[rr, "c"] \arrow[d, "d"'] && C \\
		D \arrow[d, "e"'] \arrow[rr, "\LEnv{d}{c}"] && C \\
		E \arrow[urr, "\LEnv{ed}{c}"', dotted]
	\end{tikzcd}, \quad 
	\begin{tikzcd}
		A \arrow[rr, "c"] \arrow[d, "d"'] && C \\
		D \arrow[d, "e"'] \arrow[rr, "\REnv{d}{c}"] && C \\
		E \arrow[urr, "\REnv{ed}{c}"', dotted]
	\end{tikzcd}\]
	\begin{enumerate} [(i)]
		\item If $C$ is complete and $e$ has a pre right adjoint, then the commutation on the left holds, i.e., $\LEnv{ed}{c} \circ e = \LEnv{d}{c}$.
		\item If $C$ is co-complete and $e$ has a post right adjoint, then the commutation on the right holds, i.e., $\REnv{ed}{c} \circ e = \REnv{d}{c}$.
	\end{enumerate}
\end{lemma}

\begin{proof} We only prove Claim (ii), as Claim (i) is analogous. Since $C$ is co-complete, it is sufficient to show that the two functors are equal pointwise. So fixing an arbitrary $x\in D$ we have to show that $\REnv{ed}{c}(e(x)) = \REnv{d}{c}(x)$. We have the following commutation diagram
	\[\begin{tikzcd}
		A \arrow[d, "="'] \arrow[r, "d"] & D \arrow[d, "e"] & \star \arrow[l, "x"'] \arrow[d, "="] \\
		A \arrow[r, "ed"] & E & \star \arrow[l, "e(x)"']
	\end{tikzcd} \imply
	\begin{tikzcd} \Comma{de}{e(x)} \arrow[d, "e_*"] \\ \Comma{d}{x} \end{tikzcd}\]
	The projection of this induced functor along the first components leads to the following commuting diagram :
	\[\begin{tikzcd} [column sep = large]
		\Comma{ed}{e(x)} \arrow[r, "e_*"] \arrow[dr, bend right=20, "\Forget_1^{(de)}"', pos=0.3] & \Comma{d}{x} \arrow[d, "\Forget_1^{(d)}"'] \arrow[dr, bend left=20, "", dashed] \\
		& A \arrow[r, "c"] & C
	\end{tikzcd}\]
	Then we have
	\[ \REnv{ed}{c}(e(x)) = \colim c \Forget_1^{(de)} = \colim c \Forget_1^{(d)} e_* = \colim c \Forget_1 = \REnv{d}{c}(x). \]
	The first and fourth equalities follow from Lemma \ref{lem:Kan_pointwise}, the second from the commutation. Since $e$ has a post-right adjoint, so does $e_*$ by Lemmas \ref{lem:preadj:2}. The third equality now follows from \ref{lem:preadj:1}~(i). This completes the proof of Lemma \ref{lem:Kan_id:3}.
\end{proof}

\begin{lemma} \label{lem:Kan_id:4}
	Consider the following arrangement of functors $a, b, d, f, e$ as shown below. 
	\[\begin{tikzcd}
		A \arrow[r, "a"] \arrow[d, "d"'] & B \arrow[rr, "b"] \arrow[d, "f"'] && C \\
		D \arrow[rrr, bend right=20, "\LEnv{d}{ba}"', dotted] \arrow[r, "e"] & E \arrow[rr, "\LEnv{f}{b}"] && C
	\end{tikzcd}, \quad 
	\begin{tikzcd}
		A \arrow[r, "a"] \arrow[d, "d"'] & B \arrow[rr, "b"] \arrow[d, "f"'] && C \\
		D \arrow[rrr, bend right=20, "\REnv{d}{ba}"', dotted] \arrow[r, "e"] & E \arrow[rr, "\REnv{f}{b}"] && C
	\end{tikzcd}\]
	\begin{enumerate} [(i)]
		\item If $C$ is complete and $e$ and $a$ have pre right adjoints, then the commutation on the left holds.
		\item If $C$ is co-complete and $e$ and $a$ have post right adjoints, then the commutation on the right holds.
	\end{enumerate}
\end{lemma}

The commutation follows from :
\[\begin{split}
	\REnv{d}{ba} &= \REnv{ed}{ba} e \quad \mbox{ by Lemma \ref{lem:Kan_id:3}}~(ii) , \\
	&= \REnv{fa}{ba} e \\
	&= \REnv{f}{b} e \quad \mbox{ by Lemma \ref{lem:Kan_id:2}}~(i) .
\end{split}\]
This completes the proof of Lemma \ref{lem:Kan_id:4}. \qed 

\subsection{Proof of Theorem \ref{thm:null_ext} } \label{sec:proof:null_ext}

Claim (iii) is a direct consequence of claims (i) and (ii), so we just prove the first two claims. The following lemma will come useful :

\begin{lemma} \label{lem:null_ext:5}
	Let the same assumptions as Theorem \ref{thm:null_ext} hold. Then :
	\begin{enumerate} [(i)]
		\item The functors $\pi_*$ from \eqref{eqn:def:pi_star} and $\bot_2$ have right inverses.
		\item The functors $\incl_1$ from \eqref{eqn:null_ext:1} and $\incl_2$ from \eqref{eqn:null_ext:2} have right inverses.
	\end{enumerate}
\end{lemma}

Lemma \ref{lem:null_ext:5} is proved in Section \ref{sec:proof:null_ext:5}. A key technical result to be presented later is Lemma \ref{lem:null_ext:5}. The proof of diagram \eqref{eqn:null_ext:3} requires a total of 6 commutation relations to be verified. 
\begin{enumerate}
	\item This first three are the three commutation squares along the left column of the figure
	\begin{equation} \label{eqn:null_ext:4}
		\begin{tikzcd}
			\Comma{\gamma}{\gamma} \arrow[r] & \Comma{\gamma j_1 j_2}{\gamma} \\
			\Comma{\base}{\base} \arrow[r] \arrow[u] & \Comma{j_1 j_2}{\main} \arrow[u]
		\end{tikzcd}, \quad 
		\begin{tikzcd}
			\Comma{\base}{\base} \arrow[d, "="'] \arrow[r, "\incl_2"] & \Comma{j_1 j_2}{\main} \arrow[d, "\pi_*"] \\
			\Comma{\base}{\base} \arrow[r, "\incl_1"] & \Comma{j_2}{\pi}
		\end{tikzcd}, \quad 
		\begin{tikzcd}
			\Comma{\base}{\base} \arrow[d, "\Forget_2"'] \arrow[r, "\incl_1"] & \Comma{j_2 \pi}{} \arrow[d, "\Forget_2"] \\
			\base \arrow[r, "j_1 j_2"] & \main
		\end{tikzcd} .
	\end{equation}
	The diagrams in \eqref{eqn:null_ext:4} are proved in Section \ref{sec:proof:null_ext:4}.
	\item Upper commuting triangle (yellow) : The two nullity structures are defined by composition. So their commutation holds due to the two commutations in the row above.
	\item Middle commuting triangle (blue) : This is an identity involving two left envelopes with the same codomain. The commutation follows from Lemma \ref{lem:null_ext:5}~(ii) and Lemma \ref{lem:Kan_id:4} (i) .
	\item Lower commuting triangle (red) : This commutation is a natural transformation according to Claim (i), or an isomorphism according to Claim (ii). 
\end{enumerate}

\paragraph{Proof of Claim (i)} By definition of envelopes, we have $\LEnv{\pi_*}{ \Null_{\inter, \base} } \Rightarrow \REnv{ \Forget_2 }{ \LEnv{\pi_*}{ \Null_{\inter, \base} } } \Forget_2$. This means that
\[\begin{split}
	\Null_{\base, \base} &= \LEnv{\pi_*}{ \Null_{\inter, \base} } \incl_1 \\
	& \Rightarrow \REnv{ \Forget_2 }{ \LEnv{\pi_*}{ \Null_{\inter, \base} } } \Forget_2 \incl_1 \\
	&= \REnv{ \Forget_2 }{ \LEnv{\pi_*}{ \Null_{\inter, \base} } } j_1 j_2 \Forget_2^{\base}  \\
	&= \left[ \REnv{ \Forget_2 }{ \LEnv{\pi_*}{ \Null_{\inter, \base} } } j_1 j_2 \right] \Forget_2^{\base} .
\end{split}\]
Also note that 
\[ \Null_{\base, \base} \Rightarrow \REnv{\Forget_2^{\base}}{\Null_{\base, \base}} \Forget_2^{\base} = \bar{\Null} \Forget_2^{\base} . \]
By the universality of $\bar{\Null}$ this means that 
\[ \bar{\Null} \Rightarrow \REnv{ \Forget_2 }{ \LEnv{\pi_*}{ \Null_{\inter, \base} } } j_1 j_2 , \]
as claimed.

\paragraph{Proof of Claim (ii)} Now suppose that Assumption \ref{A:5} holds. Assumption \ref{A:5} implies that given any $\calB$ object $X$, the left slice of $j_2(X)$ under $j_2$ coincides with the left slice of $X$ in $\calB$. Now take any object $X$. Then $\overline{\Null^{(\base)}}(X)$ is the colimit of the composite functor $\LEnv{\pi_*}{\Null_{\inter, \base}} \circ \incl_1$ applied the left slice $\Comma{\base}{X}$ of $X$ in $\base$. Thus it is the colimit of the functor $\LEnv{\pi_*}{\Null_{\inter, \base}}$ applied to the image of $\Comma{\base}{X}$ under $\incl_1$. By Assumption \ref{A:5} this image is precisely the left slice of $j_1 j_2(X)$ in $\Comma{j_2}{\pi}$. Thus $\overline{\Null^{(\base)}}(X)$ is the colimit of the functor $\LEnv{\pi_*}{\Null_{\inter, \base}}$ applied to the left slice of $j_1 j_2(X)$ in $\Comma{j_2}{\pi}$. But the latter is just $\Null^{(\main)} \paran{j_1 j_2(X)}$, proving the claim.
This completes the proof of Theorem \ref{thm:null_ext}. \qed

\subsection{Proof of \eqref{eqn:null_ext:4}} \label{sec:proof:null_ext:4}

The first commutation square comprises of comma categories and induced functors. The commutation follows from the following commutative diagram.
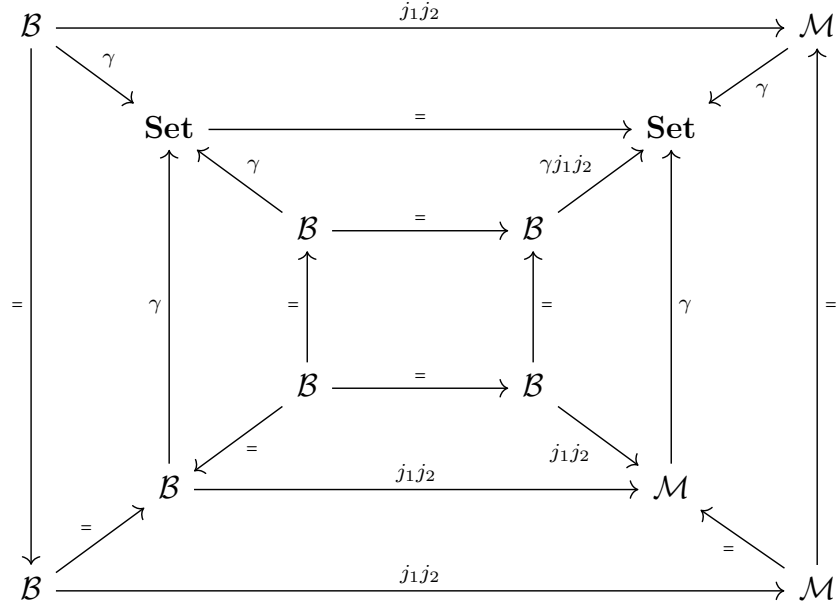
\[\begin{tikzcd}
	\base \arrow[dr, "\gamma"] \arrow[rrrrrr, "j_1 j_2" ] \arrow[dddddd, "="' ] & & & {} & & & \main \arrow[dl, "\gamma"] \\
	& \SetCat \arrow[rrrr, "=" ] & & {} & & \SetCat & \\
	& & \base \arrow[ul, "\gamma"'] \arrow[rr, "=" ] & {} & \base \arrow[ur, "\gamma j_1 j_2"] & & \\
	\\
	& & \base \arrow[uu, "=" ] \arrow[rr, "=" ] \arrow[dl, "="] & {} & \base \arrow[uu, "="' ] \arrow[dr, "j_1 j_2"'] & & \\
	& \base \arrow[uuuu, "\gamma" ] \arrow[rrrr, "j_1 j_2" ] & & {} & & \calM \arrow[uuuu, "\gamma"' ] & \\ 
	\base \arrow[ur, "="] \arrow[rrrrrr, "j_1 j_2" ] & & & {} & & & \main \arrow[uuuuuu, "="' ] \arrow[ul, "="] 
\end{tikzcd}\]
Each of the four corners represent the constituents of the four comma categories involved. Lemma \ref{lem:oh9d} applies to each of the triples of functors to provide the final commutation between comma categories. Similarly, the second commutation square in \eqref{eqn:null_ext:4} follows from the following deconstructive commutative diagram :
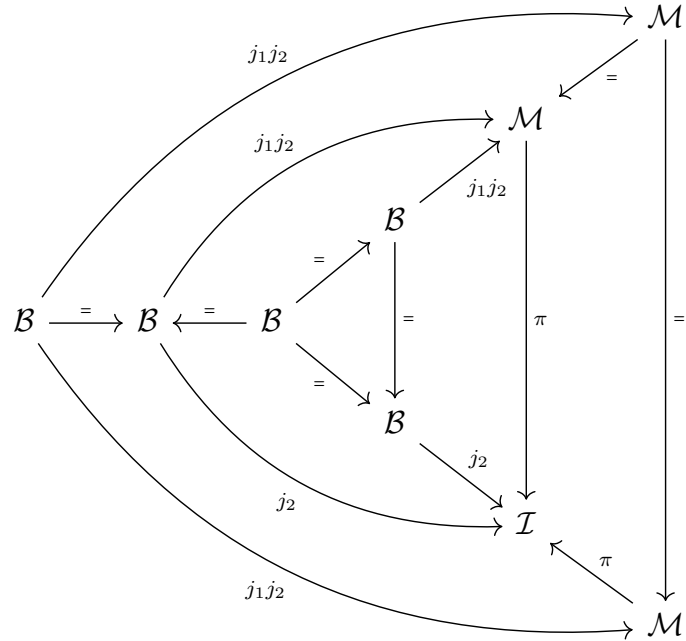
\[\begin{tikzcd}
	& & & & & \main \arrow[ld, "="] \arrow[dddddd, "="] \\
	& & & & \main \arrow[dddd, "\pi"] & \\
	& & & \base \arrow[ru, "j_1 j_2"'] \arrow[dd, "="] & & \\
	\base \arrow[r, "="] \arrow[rrrrrddd, "j_1 j_2"', bend right] \arrow[rrrrruuu, "j_1 j_2", bend left] & \base \arrow[rrrdd, "j_2"', bend right] \arrow[rrruu, "j_1 j_2", bend left] & \base \arrow[l, "="'] \arrow[ru, "="] \arrow[rd, "="'] & & & \\
	& & & \base \arrow[rd, "j_2"] & & \\
	& & & & \inter & \\
	& & & & & \main \arrow[lu, "\pi"'] 
\end{tikzcd}\]
The final commutation square in \eqref{eqn:null_ext:4} is a direct application of \eqref{eqn:dgbw}. This completes the proof of \eqref{eqn:null_ext:4}. \qed

\subsection{Proof of Lemma \ref{lem:null_ext:5}} \label{sec:proof:null_ext:5}

Lemma \ref{lem:oh9d} will be used to prove the right inverses of the induced functors between comma categories. This is done via the following diagram :
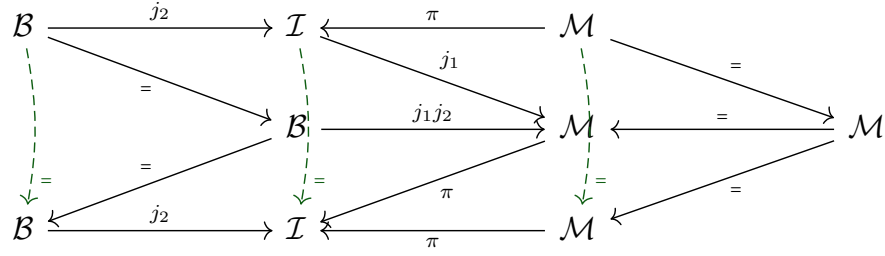
\[\begin{tikzcd} [column sep = large]
	\base \arrow[rr, "j_2"] \arrow[rrd, "="'] \arrow[dd, dashed, Shobuj, pos=0.8, "=", bend left=10] && \inter \arrow[rrd, "j_1"] \arrow[dd, dashed, Shobuj, pos=0.8, "=", bend left=10] && \main \arrow[ll, "\pi"'] \arrow[rrd, "="] \arrow[dd, dashed, Shobuj, pos=0.8, "=", bend left=10] && \\
	&& \base \arrow[lld, "="'] \arrow[rr, "j_1 j_2"] && \main \arrow[lld, "\pi"] && \main \arrow[lld, "="] \arrow[ll, "="'] \\
	\base \arrow[rr, "j_2"] && \inter && \main \arrow[ll, "\pi"] && 
\end{tikzcd}\]
Each row represents the constituents of a comma category. The triple of arrows from the second to the third row constitute $\pi_*$. The triple of arrows created by composition, from the first to the third row are identities. This makes the triple of arrows from the first to the second layer a right inverse to $\pi_*$. The right inverse of $\bot_2$ follows in a similar manner from the following diagram :
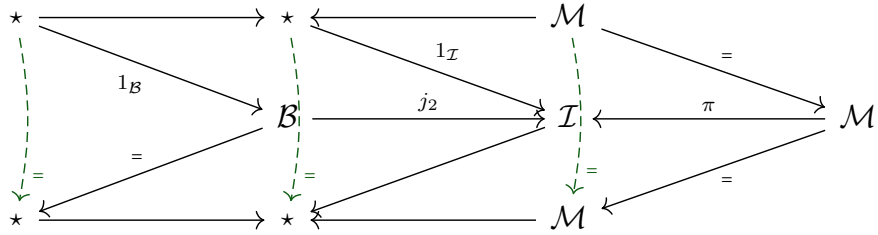
\[\begin{tikzcd} [column sep = large]
	\star \arrow[rr] \arrow[rrd, "1_{\base}"'] \arrow[dd, dashed, Shobuj, pos=0.8, "=", bend left=10] && \star \arrow[rrd, "1_{\inter}"] \arrow[dd, dashed, Shobuj, pos=0.8, "=", bend left=10] && \main \arrow[ll] \arrow[rrd, "="] \arrow[dd, dashed, Shobuj, pos=0.8, "=", bend left=10] && \\
	&& \base \arrow[lld, "="'] \arrow[rr, "j_2"] && \inter \arrow[lld] && \main \arrow[lld, "="] \arrow[ll, "\pi"'] \\
	\star \arrow[rr] && \star && \main \arrow[ll] && 
\end{tikzcd}\]
Here $\star$ denotes the trivial 1-point category. We have used the trivial fact that the comma category created out of $\star \rightarrow \star \leftarrow \calC$ is isomorphic to $\calC$. This completes the proof of Claim (i).

Next we construct a pre- right adjoint to $\incl_1$. It requires taking note of the following induced functors :
\[\begin{tikzcd}
	\base \arrow[r, "j_2"] \arrow[d, "="] & \inter \arrow[d, "="] & \main \arrow[l, "\pi"'] \arrow[d, "\pi"] \\
	\base \arrow[r, "j_2"'] & \inter & \inter \arrow[l, "="]
\end{tikzcd} \imply 
\begin{tikzcd} \Comma{j_2}{\pi} \arrow[d, "\incl_4", Shobuj] \\ \Comma{j_2}{\inter} \end{tikzcd} ; \quad 
\begin{tikzcd}
	\base \arrow[r, "j_2"'] \arrow[d, "="'] & \inter \arrow[d, "="'] & \inter \arrow[d, "j_1"] \arrow[l, "="] \\
	\base \arrow[r, "j_2"'] & \inter & \main \arrow[l, "\pi"]
\end{tikzcd} \imply 
\begin{tikzcd} \Comma{j_2}{\inter} \arrow[d, "\incl_5", Shobuj] \\ \Comma{j_2}{\pi} \end{tikzcd}
\]
Now $\incl_1$, $\incl_3$, $\incl_3^{R*}$, $\incl_4$, $\incl_5$ assemble together to produce a right inverse to $\incl_1$.
\[\begin{tikzcd}
	\Comma{j_2}{\pi} \arrow[dr, "="'] \arrow[r, "\incl_4"] & \Comma{j_2}{\inter} \arrow[d, "\incl_5"] \\
	& \Comma{j_2}{\pi}
\end{tikzcd} ; \quad 
\begin{tikzcd}
	& \Comma{j_2}{\inter} \arrow[dd, "\incl_3^{R*}"'] & \Comma{j_2}{\pi} \\
	\Comma{j_2}{\pi} \arrow[ur, "\incl_4"] \arrow[dr, dashed, "\incl_1^{R*}"'] & \\
	& \ArrowCat{\base} \arrow[uur, dashed, "\incl_1"'] \arrow[r, "\incl_3"] & \Comma{j_2}{\inter} \arrow[uu, "\incl_5"']
\end{tikzcd}\]
The commutation says that
\[ \incl_1 \incl_1^R = \incl_5 \incl_3 \incl_3^R \incl_4 = \incl_5 \incl_4 = \Id, \]
proving the claim. Similarly we have two other induced functor between comma categories :
\[\begin{tikzcd}
	\base \arrow[r, "j_2"'] \arrow[d, "="'] & \inter \arrow[d, "j_1"'] & \inter \arrow[d, "j_1"] \arrow[l, "="] \\
	\base \arrow[r, "j_1 j_2"'] & \main & \main \arrow[l, "="]
\end{tikzcd} \imply 
\begin{tikzcd} \Comma{j_2}{\inter} \arrow[d, "\incl_6", Shobuj] \\ \Comma{j_1 j_2}{\main} \end{tikzcd} ; \quad
\begin{tikzcd}
	\base \arrow[r, "j_1 j_2"] \arrow[d, "="] & \main \arrow[d, "="] & \main \arrow[l, "="'] \arrow[d, "\pi"] \\
	\base \arrow[r, "j_2"'] & \inter & \inter \arrow[l, "="]
\end{tikzcd} \imply 
\begin{tikzcd} \Comma{j_1 j_2}{\main} \arrow[d, "\incl_7", Shobuj] \\ \Comma{j_2}{\inter} \end{tikzcd}\]
Now $\incl_2$, $\incl_3$, $\incl_3^{R*}$, $\incl_6$, $\incl_7$ assemble together to produce a right inverse to $\incl_2$.
\[\begin{tikzcd}
	\Comma{j_1 j_2}{\main} \arrow[dr, "="'] \arrow[r, "\incl_7"] & \Comma{j_2}{\inter} \arrow[d, "\incl_6"] \\
	& \Comma{j_1 j_2}{\main}
\end{tikzcd} ; \quad 
\begin{tikzcd}
	& \Comma{j_2}{\inter} \arrow[dd, "\incl_3^{R*}"'] & \Comma{j_1 j_2}{\main} \\
	\Comma{j_1 j_2}{\main} \arrow[ur, "\incl_7"] \arrow[dr, dashed, "\incl_2^{R*}"'] & \\
	& \ArrowCat{\base} \arrow[uur, dashed, "\incl_2"'] \arrow[r, "\incl_3"] & \Comma{j_2}{\inter} \arrow[uu, "\incl_6"']
\end{tikzcd}\]
This completes the proof of Lemma \ref{lem:null_ext:5}. \qed


\begin{thebibliography}{10}

\bibitem{Hornik1990univ}
K.~Hornik, M.~Stinchcombe, and H.~White.
\newblock \href{http://dx.doi.org/10.1016/0893-6080(90)90005-6}{Universal
  approximation of an unknown mapping and its derivatives using multilayer
  feedforward networks}.
\newblock {\em Neural networks}, 3(5):551--560, 1990.

\bibitem{HornikStinchcombeWhite1989multi}
K.~Hornik, M.~Stinchcombe, and H.~White.
\newblock \href{http://dx.doi.org/10.1016/0893-6080(89)90020-8}{Multilayer
  feedforward networks are universal approximators}.
\newblock {\em Neural networks}, 2(5):359--366, 1989.

\bibitem{ChenChen1995universal}
T.~Chen and H.~Chen.
\newblock \href{http://dx.doi.org/10.1109/72.392253}{Universal approximation to
  nonlinear operators by neural networks with arbitrary activation functions
  and its application to dynamical systems}.
\newblock {\em IEEE Trans. Neural Networks}, 6(4):911--917, 1995.

\bibitem{Das2023Lie}
S.~Das.
\newblock \href{http://dx.doi.org/10.1088/1361-6544/acc22c}{Lie group valued
  {K}oopman eigenfunctions}.
\newblock {\em Nonlinearity}, 36:2149–2165, 2023.

\bibitem{DasGiannakis2023harmonic}
S.~Das and D.~Giannakis.
\newblock \href{http://dx.doi.org/10.1007/s00041-023-09992-4}{On harmonic
  {H}ilbert spaces on compact abelian groups}.
\newblock {\em J. Fourier Anal. Appl.}, 29(1):12, 2023.

\bibitem{DGJ_compactV_2018}
D.~Giannakis, S.~Das, and J.~Slawinska.
\newblock \href{http://dx.doi.org/10.1016/j.acha.2021.02.004}{Reproducing
  kernel {H}ilbert space compactification of unitary evolution groups}.
\newblock {\em Appl. Comput. Harmon. Anal.}, 54:75--136, 2021.

\bibitem{PalisSmale1970structural}
J.~Palis and S.~Smale.
\newblock \href{https://bookstore.ams.org/pspum-14}{Structural stability
  theorems}.
\newblock \href{https://bookstore.ams.org/pspum-14}{In {\em Global analysis}},
  volume~14, pages 223--231, 1970.

\bibitem{HirschEtAl1970ngbr}
M.~Hirsch et~al.
\newblock \href{http://dx.doi.org/10.1007/BF01404552}{Neighborhoods of
  hyperbolic sets}.
\newblock {\em Inventiones mathematicae}, 9(2):121--134, 1970.

\bibitem{kieffer1980coding}
J.~Kieffer.
\newblock \href{www.jstor.org/stable/2243064}{On coding a stationary process to
  achieve a given marginal distribution}.
\newblock {\em Ann. Prob.}, 1980.

\bibitem{AlpernPrasad2005towers}
S.~Alpern and VS. Prasad.
\newblock \href{www.cdam.lse.ac.uk/Reports/Files/cdam-2005-21.pdf}{Towers,
  conjugacy and coding}.
\newblock \href{www.cdam.lse.ac.uk/Reports/Files/cdam-2005-21.pdf}{In {\em
  Centre for Discrete and Applicable Mathematics}}. London School of Economics,
  2005.

\bibitem{Das2023Koop_susp}
S.~Das.
\newblock
  \href{www.sciencedirect.com/science/article/abs/pii/S0022039625008137}{Smooth
  {K}oopman eigenfunctions}.
\newblock {\em J. Diff. Eq.}, 453(1), 2026.

\bibitem{Jewett1970prevalence}
R.~Jewett.
\newblock \href{https://www.jstor.org/stable/24901717}{The prevalence of
  uniquely ergodic systems}.
\newblock {\em J. Math. Mech.}, 19(8):717--729, 1970.

\bibitem{HSYprevalence1992}
B.~Hunt, T.~Sauer, and J.~Yorke.
\newblock \href{http://dx.doi.org/10.1090/S0273-0979-1992-00328-2}{Prevalence:
  a translation-invariant “almost every” on infinite-dimensional spaces}.
\newblock {\em Bull. Amer. Math. Soc.}, 27(2):217--238, 1992.

\bibitem{DasYorke2020}
S.~Das and J.~Yorke.
\newblock \href{http://dx.doi.org/10.1007/s11071-020-05590-x}{Crinkled changes
  of variables}.
\newblock {\em Non. Dyn.}, 102:645--652, 2020.

\bibitem{DasSaddles2015}
S.~Das.
\newblock \href{http://topology.auburn.edu/tp/reprints/v47/tp47012p1.pdf}{Dense
  saddles in torus maps}.
\newblock {\em Topology Proc.}, 47:177--190, 2015.

\bibitem{DasJim2017chaos}
S.~Das and J.~Yorke.
\newblock \href{http://dx.doi.org/10.1137/17M1113199}{Multichaos from
  quasiperiodicity}.
\newblock {\em SIAM J. Appl. Dyn. Syst.}, 16(4):2196–2212, 2017.

\bibitem{Arnold1965}
V.~Arnold.
\newblock Small denominators. i. mapping of the circumference onto itself.
\newblock {\em Amer. Math. Soc. Transl. (2)}, 46:213--284, 1965.

\bibitem{Arnold1963small}
V.~Arnol'd.
\newblock Small denominators and problems of stability of motion in classical
  and celestial mechanics.
\newblock {\em Russian Mathematical Surveys}, 18(6):85, 1963.

\bibitem{Herman1}
M.~Herman.
\newblock {\em Mesure de Lebesgue et nombre de rotation}, volume 597.
\newblock Springer, 1979.

\bibitem{Das2018Tongue}
S.~Das.
\newblock
  \href{http://topology.auburn.edu/tp/reprints/v52/tp52013p1.pdf}{Universal
  bound on the measure of periodic windows of parameterized circle maps}.
\newblock {\em Topology proc.}, 52:179--187, 2018.

\bibitem{OttYorke_prevalence_2005}
W.~Ott and J.~Yorke.
\newblock \href{http://dx.doi.org/10.1090/S0273-0979-05-01060-8}{Prevalence}.
\newblock {\em Bull. Amer. Math. Soc.}, 42(3):263--290, 2005.

\bibitem{SauerEtAl1991}
T.~Sauer, J.~A. Yorke, and M.~Casdagli.
\newblock \href{http://dx.doi.org/10.1007/bf01053745}{Embedology}.
\newblock {\em J. Stat. Phys.}, 65(3--4):579--616, 1991.

\bibitem{HuntKaloshin1997proj}
B.~Hunt and V.~Kaloshin.
\newblock \href{http://dx.doi.org/10.1088/0951-7715/10/5/002}{How projections
  affect the dimension spectrum of fractal measures}.
\newblock {\em Nonlinearity}, 10(5):1031, 1997.

\bibitem{sontag2003differential}
Sontag.
\newblock \href{http://dx.doi.org/10.1007/s00332-002-0506-0}{For differential
  equations with r parameters, 2 r+ 1 experiments are enough for
  identification}.
\newblock {\em J. Nonlinear Sci.}, 12:553--583, 2003.

\bibitem{fraysse2006smooth}
A.~Fraysse and S.~Jaffard.
\newblock \href{http://dx.doi.org/10.4171/RMI/469}{How smooth is almost every
  function in a sobolev space?}
\newblock {\em Revista Matematica Iberoamericana}, 22(2):663--682, 2006.

\bibitem{HuntKaloshin1999reg}
B.~Hunt and V.~Kaloshin.
\newblock \href{http://dx.doi.org/10.1088/0951-7715/12/5/303}{Regularity of
  embeddings of infinite-dimensional fractal sets into finite-dimensional
  spaces}.
\newblock {\em Nonlinearity}, 12(5):1263, 1999.

\bibitem{Kaloshin1997prevalence}
V.~Kaloshin.
\newblock \href{http://dx.doi.org/10.1007/bf02466014}{Prevalence in the space
  of finitely smooth maps}.
\newblock {\em Functional Anal. Appli.}, 31(2), 1997.

\bibitem{Hunt1994dffrntl}
B.~Hunt.
\newblock \href{http://dx.doi.org/10.1090/S0002-9939-1994-1260170-X}{The
  prevalence of continuous nowhere differentiable functions}.
\newblock {\em Proc. Amer. Math. Soc.}, 122(3):711--717, 1994.

\bibitem{BerryDas2023learning}
T.~Berry and S.~Das.
\newblock \href{http://dx.doi.org/10.1137/22M1516865}{Learning theory for
  dynamical systems}.
\newblock {\em SIAM J. Appl. Dyn.}, 22:2082 -- 2122, 2023.

\bibitem{BerryDas2024review}
T.~Berry and S.~Das.
\newblock
  \href{https://epubs.siam.org/doi/abs/10.1137/24M1696974?journalCode=siread}{Limits
  of learning dynamical systems}.
\newblock {\em SIAM review}, 16, 2025.

\bibitem{Stark1999delay}
J.~Stark.
\newblock \href{http://dx.doi.org/10.1007/s003329900072}{Delay embeddings for
  forced systems. i. {D}eterministic forcing}.
\newblock {\em J. Non. Sc.}, 9(3):255--332, 1999.

\bibitem{Stark1999graphs}
J.~Stark.
\newblock \href{http://dx.doi.org/10.1017/S0143385799126555}{Regularity of
  invariant graphs for forced systems}.
\newblock {\em Erg. Theory Dyn. Sys.}, 19(1):155--199, 1999.

\bibitem{Das2023CatEntropy}
S.~Das.
\newblock \href{https://arxiv.org/pdf/2301.09205.pdf}{The categorical basis of
  dynamical entropy}.
\newblock {\em Applied Categorical Structures}, 32, 2024.

\bibitem{Das2024hmlgy}
S.~Das.
\newblock \href{doi.org/10.48550/arXiv.2404.03735}{Homology and homotopy for
  arbitrary categories}, 2024.

\bibitem{BauerLesnick2020prsstnc}
U.~Bauer and M.~Lesnick.
\newblock \href{http://dx.doi.org/10.1007/978-3-030-43408-3_3}{Persistence
  diagrams as diagrams: A categorification of the stability theorem}.
\newblock {\em Topological Data Analysis}, 2020.

\bibitem{fritz2020Stoch}
T.~Fritz.
\newblock \href{http://dx.doi.org/10.1016/j.aim.2020.107239}{A synthetic
  approach to {M}arkov kernels, conditional independence and theorems on
  sufficient statistics}.
\newblock {\em Adv. Math.}, 370:107239, 2020.

\bibitem{FritzEtAl2023repr}
T.~Fritz, T.~Gonda, P.~Perrone, and E.~Rischel.
\newblock \href{http://dx.doi.org/10.1016/j.tcs.2023.113896}{Representable
  {M}arkov categories and comparison of statistical experiments in categorical
  probability}.
\newblock {\em Theoretical Computer Science}, 961:113896, 2023.

\bibitem{Shiebler2020cluster}
D.~Shiebler.
\newblock Functorial clustering via simplicial complexes.
\newblock In {\em NeurIPS 2020 Workshop on Topological Data Analysis and
  Beyond}, 2020.

\bibitem{Shiebler2022kan}
D.~Shiebler.
\newblock \href{http://dx.doi.org/10.48550/arXiv.2203.09018}{Kan extensions in
  data science and machine learning}, 2022.

\bibitem{Lambek1989logic}
J.~Lambek.
\newblock \href{http://dx.doi.org/10.1007/BF00370824}{On some connections
  between logic and category theory}.
\newblock {\em Studia Logica}, 48:269--278, 1989.

\bibitem{Awodey1996}
S.~Awodey.
\newblock \href{http://dx.doi.org/10.1093/philmat/4.3.209}{Structure in
  mathematics and logic: A categorical perspective}.
\newblock {\em Philos. Math.}, 4(3):209--237, 1996.

\bibitem{wiener1921average}
N.~Wiener.
\newblock The average of an analytic functional and the brownian movement.
\newblock {\em Proc. Nat. Acad. Sci.}, 7(10):294--298, 1921.

\bibitem{CameronMartin1944Wiener}
R.~Cameron and W.~Martin.
\newblock \href{http://dx.doi.org/10.2307/1969276}{Transformations of {W}einer
  integrals under translations}.
\newblock {\em Annals of Mathematics}, 45(2):386--396, 1944.

\bibitem{CameronMartin1945Wiener}
R.~Cameron and W.~Martin.
\newblock \href{http://dx.doi.org/10.2307/1990282}{Transformations of wiener
  integrals under a general class of linear transformations}.
\newblock {\em Transactions of the American Mathematical Society},
  58(2):184--219, 1945.

\bibitem{Das2024slice}
S.~Das.
\newblock \href{http://dx.doi.org/10.64700/mmm.76}{Functors induced by comma
  categories}.
\newblock {\em Modern Math. Meth.}, 3, 2025.

\bibitem{DasSuda2024recon}
S.~Das and T.~Suda.
\newblock \href{doi.org/10.48550/arXiv.2412.19734}{Dynamics, data and
  reconstruction}, 2024.

\bibitem{DasSuda2025enrich}
S.~Das and T.~Suda.
\newblock \href{https://arxiv.org/pdf/2509.05900}{Dynamical systems as enriched
  functors}, 2025.

\bibitem{perrone2022kan}
P.~Perrone and W.~Tholen.
\newblock \href{http://dx.doi.org/10.1007/s10485-021-09671-9}{Kan extensions
  are partial colimits}.
\newblock {\em Applied Categorical Structures}, 30(4):685--753, 2022.

\bibitem{street2004categorical}
R.~Street.
\newblock
  \href{http://dx.doi.org/10.1023/B:APCS.0000049317.24861.36}{Categorical and
  combinatorial aspects of descent theory}.
\newblock {\em Applied Categorical Structures}, 12:537--576, 2004.

\bibitem{Riehl_homotopy_2014}
E.~Riehl.
\newblock {\em Categorical homotopy theory}, volume~24.
\newblock Cambridge University Press, 2014.

\bibitem{HeunenEtAl2017cnvnt}
C.~Heunen et~al.
\newblock \href{http://dx.doi.org/10.1109/LICS.2017.8005137}{A convenient
  category for higher-order probability theory}.
\newblock \href{http://dx.doi.org/10.1109/LICS.2017.8005137}{In {\em 2017 32nd
  Annual ACM/IEEE Symposium on Logic in Computer Science (LICS)}}, pages 1--12.
  IEEE, 2017.

\bibitem{Chen1977iterated}
Kuo-Tsai Chen.
\newblock \href{http://dx.doi.org/10.1090/S0002-9904-1977-14320-6}{Iterated
  path integrals}.
\newblock {\em Bull. Amer. Math. Soc.}, 83(5):831--879, 1977.

\bibitem{Stacey2008smooth}
A.~Stacey.
\newblock \href{http://dx.doi.org/10.48550/arXiv.0802.2225}{Comparative
  smootheology}, 2008.

\bibitem{BaezHoffnung2011cnvnnt}
J.~Baez and A.~Hoffnung.
\newblock \href{http://dx.doi.org/10.1090/S0002-9947-2011-05107-X}{Convenient
  categories of smooth spaces}.
\newblock {\em Trans. Amer. Math. Soc.}, 363(11):5789--5825, 2011.

\bibitem{gill2013ordinary}
T.~Gill et~al.
\newblock On ordinary and standard" lebesgue measures" in separable banach
  spaces.
\newblock {\em Georgian International Journal of Science, Technology and
  Medicine}, 5(3/4):115, 2013.

\bibitem{Pantsulaia2008generators}
G.~Pantsulaia.
\newblock \href{http://dml.mathdoc.fr/item/05309663}{On generators of shy sets
  on {P}olish topological vector spaces}.
\newblock {\em New York J. Math}, 14:235--261, 2008.

\bibitem{vershik2007does}
A.~Vershik.
\newblock \href{http://dx.doi.org/10.1134/S0081543807040153}{Does there exist a
  {L}ebesgue measure in the infinite-dimensional space?}
\newblock {\em Proc. Steklov Inst. Math.}, 259(1):248--272, 2007.

\bibitem{kakutani1950construction}
S.~Kakutani and J.~Oxtoby.
\newblock \href{http://dx.doi.org/10.2307/1969435}{Construction of a
  non-separable invariant extension of the lebesgue measure space}.
\newblock {\em Annals Math.}, 52(3):580--590, 1950.

\bibitem{baker1991lebesgue}
R.~Baker.
\newblock \href{http://dx.doi.org/10.1090/S0002-9939-1991-1062827-X}{Lebesgue
  measure on r-infinity}.
\newblock {\em Proc. Amer. Math. Soc.}, 113(4):1023--1029, 1991.

\end{thebibliography}

\end{document}